\documentclass[10pt]{amsart}

\usepackage{amsmath, amsthm}
\usepackage{eucal}

\usepackage{amssymb}
\usepackage{amscd}
\usepackage{palatino}
\usepackage{latexsym}
\usepackage{epsfig}
\usepackage{graphicx}
\usepackage{amsfonts}
\usepackage{psfrag}
\usepackage{url}
\usepackage{cite}

\usepackage[margin=1in]{geometry}

\input xy
\xyoption{all}
\UseComputerModernTips

\numberwithin{equation}{section}
\numberwithin{figure}{section}

\newtheorem{theorem}{Theorem}[section]
\newtheorem{lemma}[theorem]{Lemma}
\newtheorem{proposition}[theorem]{Proposition}
\newtheorem{corollary}[theorem]{Corollary}

\newtheorem{fact}[theorem]{Fact}

\newtheorem{remark}[theorem]{Remark}
\newtheorem{example}[theorem]{Example}
\newtheorem{question}[theorem]{Question}

\theoremstyle{definition}
\newtheorem{definition}[theorem]{Definition}

\newcommand{\C}{{\mathbb{C}}}

\newcommand{\Z}{{\mathbb{Z}}}

\newcommand{\R}{{\mathbb{R}}}
\newcommand{\N}{{\mathbb{N}}}

\renewcommand{\t}{\mathfrak{t}}

\newcommand{\into}{\hookrightarrow}

\newcommand{\myvcenter}[1]{\ensuremath{\vcenter{\hbox{#1}}}}

\DeclareMathOperator{\Lie}{Lie}

\DeclareMathOperator{\pt}{pt}

\newcommand{\hsm}{{\hspace{1mm}}}

\usepackage{multicol}
\usepackage[dvips]{color}
\definecolor{gold}{rgb}{0.85,.66,0}
\definecolor{cherry}{rgb}{0.9,.1,.2}
\definecolor{burgundy}{rgb}{0.8,.2,.2}
\definecolor{orangered}{rgb}{0.85,.3,0}
\definecolor{orange}{rgb}{0.85,.4,0}
\definecolor{olive}{rgb}{.45,.4,0}
\definecolor{lime}{rgb}{.6,.9,0}
\definecolor{green}{rgb}{.2,.7,0}
\definecolor{grey}{rgb}{.4,.4,.2}
\definecolor{brown}{rgb}{.4,.3,.1}

\newcommand{\PFill}{{\mathcal P}Fi \ell \ell}
\newcommand{\Fill}{{\mathcal F}i \ell \ell}
\newcommand{\roll}{{ro\ell \ell}}
\newcommand{\Flags}{{\mathcal F}\ell ags}
\DeclareMathOperator{\Hess}{Hess}

\begin{document}

\title{Poset pinball, highest forms, and $(n-2,2)$ Springer varieties}

\author{Barry Dewitt}
\address{Mathematical Institute\\
24--29 St Giles'\\
University of Oxford\\
Oxford, OX1 3LB \\
United Kingdom}
\email{barry.dewitt@maths.ox.ac.uk}

\author{Megumi Harada}
\address{Department of Mathematics and
Statistics\\ McMaster University\\ 1280 Main Street West\\ Hamilton, Ontario L8S4K1\\ Canada}
\email{Megumi.Harada@math.mcmaster.ca}
%\urladdr{\url{http://www.math.mcmaster.ca/Megumi.Harada/}}
\thanks{The second author is partially supported by an NSERC Discovery Grant,
an NSERC University Faculty Award, and an Ontario Ministry of Research
and Innovation Early Researcher Award.}


\date{\today}

%%%%%%%%%%%%%%%%%%%%
% Disclaimer
%%%%%%%%%%%%%%%%%%%%
% \begin{center}
% \framebox{
% {\Large\bf DRAFT (\today): DO NOT DISTRIBUTE.}}
% \end{center}

%%%%%%%%%%%%%%%%%%%%%
%  Abstract
%%%%%%%%%%%%%%%%%%%%%

\begin{abstract}

In this manuscript we study type $A$ nilpotent Hessenberg varieties
equipped with a natural $S^1$-action using techniques introduced by
Tymoczko, Harada-Tymoczko, and Bayegan-Harada, with a particular
emphasis on a special class of nilpotent Springer varieties
corresponding to the partition $\lambda= (n-2,2)$ for $n \geq
4$. First we define the \textbf{adjacent-pair matrix} corresponding to
any filling 
of a Young diagram with $n$ boxes with the alphabet
$\{1,2,\ldots,n\}$. Using the adjacent-pair matrix we make more explicit and
also extend some statements concerning \textbf{highest forms of linear operators} in previous work of
Tymoczko. Second, for a nilpotent operator $N$ and Hessenberg function
$h$, we construct an explicit bijection
between the $S^1$-fixed points of the 
nilpotent Hessenberg variety $\Hess(N,h)$ 
and the 
set of $(h,\lambda_N)$-permissible fillings of
the Young diagram $\lambda_N$. 
Third, we use \textbf{poset pinball}, the combinatorial game introduced by Harada
  and Tymoczko, to study the $S^1$-equivariant cohomology of type $A$
  Springer varieties $\mathcal{S}_{(n-2,2)}$ associated to Young
  diagrams of shape $(n-2,2)$ for $n\geq 4$. Specifically, we use the \textbf{dimension pair algorithm} for Betti-acceptable
  pinball described by Bayegan and Harada to specify a subset of the 
  equivariant Schubert classes in the $T$-equivariant cohomology of the
  flag variety $\mathcal{F}\ell ags(\C^n) \cong GL(n,\C)/B$
  which maps to a module basis of
  $H^*_{S^1}(\mathcal{S}_{(n-2,2)})$ under the projection map
  $H^*_T(\mathcal{F}\ell ags(\C^n)) \to
  H^*_{S^1}(\mathcal{S}_{(n-2,2)})$. 
 Our poset pinball module basis is not
  poset-upper-triangular; this is the first concrete such example in
  the literature. A straightforward consequence of our proof is that there exists a simple and explicit change of
  basis which transforms our poset pinball basis to a poset-upper-triangular module basis for
  $H^*_{S^1}(\mathcal{S}_{(n-2,2)})$. We close with open questions for
  future work. 
\end{abstract}

\maketitle

\setcounter{tocdepth}{1}
\tableofcontents

\section{Introduction}\label{sec:intro}

The study of Hessenberg varieties is an active field of
modern mathematical research.  Indeed, Hessenberg varieties arise in
many areas of mathematics, including geometric representation theory
\cite{Spa76, Shi85, Fun03}, numerical analysis \cite{DeMProSha92},
mathematical physics \cite{Kos96, Rie03}, combinatorics \cite{Ful99},
and algebraic geometry \cite{BriCar04, CarrellKaveh:2008}, so it is of interest
to explicitly analyze their topology, e.g. the structure of their
(equivariant) cohomology rings. In this paper we further develop
the approach, initiated and developed in \cite{HarTym09, HarTym10,
  BayHar10a, BayHar10b}, which studies the topology of Hessenberg varieties
through poset pinball and Schubert calculus techniques.

In this manuscript we focus on the case of nilpotent Hessenberg varieties, and more
particularly on nilpotent Springer varieties. 
We begin by briefly recalling the setting of our results; for more details we
refer the reader to Section~\ref{sec:springer and circle action}.
Let $N: \C^n \to \C^n$
be a nilpotent operator. Let
$h: \{1,2,\ldots, n\} \to \{1,2,\ldots,
n\}$ be a function satisfying $h(i) \geq i$ for all $1 \leq i \leq n$ and $h(i+1)
\geq h(i)$ for all $1 \leq i < n$. 
In type $A$, \textbf{nilpotent Hessenberg varieties}
can be defined as the following subvariety of
$\mathcal{F}\ell ags(\C^n)$: 
\[
\Hess(N,h) := \{ V_\bullet = (0 \subseteq V_1 \subseteq V_2 \subseteq \cdots
\subseteq V_{n-1} \subseteq V_n = \C^n) \hsm \mid \hsm  NV_i \subseteq
V_{h(i)} \textup{ for all } i = 1, \ldots, n \}.
\]
We equip $\Hess(N,h)$ with a natural $S^1$-action (described precisely in
Section~\ref{sec:springer and circle action}) induced from the
diagonal torus subgroup $T$ of $U(n,\C)$ acting in the usual fashion
on 
$GL(n,\C)/B \cong \Flags(\C^n)$. In the special case
when the Hessenberg function $h: \{1,2,\ldots,n\} \to
\{1,2,\ldots,n\}$ is the identity $h(i)=i$ for all $1 \leq i \leq n$,
we call $\Hess(N,h)$ a \textbf{nilpotent Springer variety} and denote
it by $\mathcal{S}_N$.

Our first two results apply to general type $A$ nilpotent Hessenberg
varieties. Let $N$ be a nilpotent $n \times n$ matrix in Jordan
canonical form with weakly decreasing block sizes and let $\lambda$
denote the Young diagram\footnote{We use English notation for Young
  diagrams.} (equivalently the partition) with row lengths the Jordan
block sizes of $N$ listed in weakly decreasing order. In \cite[Theorem
6.1]{Tym06}
Tymoczko builds a paving-by-affines of a nilpotent Hessenberg variety
$\Hess(N,h)$, where the nilpotent operator $N$ is required to be in  
\textbf{highest form}
(see \cite[Definition 4.1]{Tym06}). Much topological information about
a variety is encoded in a paving-by-affines, so it is useful to 
build tools for dealing with the technical condition 
that the operator $N$ be in highest form. 
Our first contribution is to
introduce what we call the \textbf{adjacency-pair matrix}, which is 
an $n \times n$ matrix constructed from a filling of a Young diagram
$\lambda$ with $n$ boxes by the alphabet $\{1,2,\ldots,n\}$. This then
allows us to 
make more explicit and also generalizes a procedure for producing
highest forms of linear operators sketched in
\cite[Section 4]{Tym06}. In particular our methods allows us to
straightforwardly derive the explicit change-of-basis permutation
matrix which puts $N$ into any choice of highest form (including that
used by Tymoczko in \cite{Tym06},
cf. Corollary~\ref{corollary:tymoczko}).
The adjacent-pair matrices
also allows us to see precisely the set of permutation matrices which
conjugate $N$ to highest form (Theorem~\ref{theorem:N_T and
  algorithm}). 
The explicit nature of our results allows for other computations related to these
nilpotent Hessenberg varieties. As an example, we derive in
Lemma~\ref{lemma:conjugated circle} an explicit formula for the Lie
algebra projection induced by the inclusion of the $S^1$ subgroup
acting on a special case of nilpotent Springer variety into the diagonal subgroup
$T$ of $U(n,\C)$ acting on $\Flags(\C^n)$.  Thus we expect our procedure to be useful for
future poset pinball analysis of type $A$ nilpotent Hessenberg
varieties.

The affine cells in Tymoczko's paving-by-affines of $\Hess(N,h)$ are
in 
one-to-one correspondence with 
\textbf{permissible fillings} of Young diagrams (defined precisely in
Section~\ref{sec:S1 action}); this is a useful combinatorial
enumeration of the affine cells. The correspondence arises since the
affine cells 
are intersections of $\Hess(N,h)$ with certain Schubert
cells $BwB \subseteq GL(n,\C)/B \cong \Flags(\C^n)$. Each such
Schubert cell contains a unique (coset of a) permutation matrix $wB$,
and each permutation $w$ can be associated to a permissible filling of
$\lambda$.  Our second contribution is to extend this relationship between
the permutations (which in this manuscript we think of as $S^1$-fixed
points of $\Hess(N,h)$) and the permissible fillings as follows. For the purpose
of the discussion below assume that $N$ is in Jordan canonical form with
weakly decreasing block sizes. We define for each permutation $\sigma
\in S_n$ a bijection $\phi_{\lambda,\sigma}$ between the set
$\Fill(\lambda)$ of fillings of $\lambda$ with the set of permutations
$S_n$ (Definition~\ref{definition:phi lambda sigma}). Each
$\phi_{\lambda,\sigma}$ then induces a bijection between the
permissible fillings $\PFill(\lambda)$ of $\lambda$ and the
$S^1$-fixed points of the \emph{translated} Hessenberg variety
$\Hess(\sigma N \sigma^{-1}, h)$. In particular this yields an explicit
formula for this bijection for all the possible highest forms of $N$
in Theorem~\ref{theorem:N_T and algorithm}. 
Our results also provide proofs of statements quoted in
\cite{BayHar10b}.

Our third contribution is an explicit construction of a
computationally convenient module basis for the $S^1$-equivariant
cohomology\footnote{We work with cohomology with
  coefficients in $\C$ throughout, and hence omit it from our
  notation.}   of a special class of 
type $A$ nilpotent Springer varieties, 
namely, the $2$-block (also known as $2$-row) nilpotent
Springer varieties associated to Young diagrams of the form
$(n-2,2)$, e.g.
\[
{\def\lr#1{\multicolumn{1}{|@{\hspace{.6ex}}c@{\hspace{.6ex}}|}{\raisebox{-.3ex}{$#1$}}}
\raisebox{-.6ex}{$\begin{array}[b]{ccccccc}
\cline{1-1}\cline{2-2}\cline{3-3}\cline{4-4}\cline{5-5}\cline{6-6}\cline{7-7}
\lr{\phantom{1}}&\lr{\phantom{1}}&\lr{\phantom{1}}&\lr{\phantom{1}}&\lr{\phantom{1}}&\lr{\phantom{1}}&\lr{\phantom{1}}
\\
\cline{1-1}\cline{2-2}\cline{3-3}\cline{4-4}\cline{5-5}\cline{6-6}\cline{7-7}
\lr{\phantom{1}}&\lr{\phantom{1}}\\
\cline{1-1}\cline{2-2}
\end{array}$}.
}
\]
Here and below we will always assume $n\geq 4$, so the smallest
Springer variety we consider corresponds to the $2 \times 2$ block
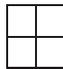
\[
{\def\lr#1{\multicolumn{1}{|@{\hspace{.6ex}}c@{\hspace{.6ex}}|}{\raisebox{-.3ex}{$#1$}}}
\raisebox{-.6ex}{$\begin{array}[b]{cc}
\cline{1-1}\cline{2-2}
\lr{\phantom{1}}&\lr{\phantom{1}}\\
\cline{1-1}\cline{2-2}
\lr{\phantom{1}}&\lr{\phantom{1}}\\
\cline{1-1}\cline{2-2}
\end{array}$}
}
\]
More specifically, we use the \textbf{poset pinball} methods introduced in
\cite{HarTym10} and the \textbf{dimension pair algorithm} for determining
pinball rolldowns described in \cite{BayHar10b} to construct our
combinatorially natural module basis for 
$H^*_{S^1}(\mathcal{S}_{(n-2,2)})$. Our arguments
use our results above on highest forms and the explicit
correspondence between permissible fillings and $S^1$-fixed points of
the Springer variety. The module basis is obtained by taking images
under the natural projection map $H^*_T(\mathcal{F}\ell ags(\C^n)) \to
H^*_{S^1}(\mathcal{S}_{(n-2,2)})$, to be described in detail below, of
a subset of the $T$-equivariant Schubert classes in
$H^*_T(\mathcal{F}\ell ags(\C^n))$.  A similar analysis by Bayegan and
the second author in a special case of regular nilpotent Hessenberg
varieties \cite{BayHar10b} yields a \emph{poset-upper-triangular}
basis in the sense of \cite{HarTym10}. In contrast to the results in
\cite{BayHar10b}, in the present manuscript we find that the module
basis is \emph{not} poset-upper-triangular; this is the first such example
in the literature.
In addition, a straightforward
consequence of our proof is that a simple change of variables yields a
module basis which is not a poset pinball basis but is
poset-upper-triangular.  These results provides further evidence for
the point of view, explained in \cite{HarTym10}, that geometrically
natural GKM-type module bases in equivariant cohomology may not always
be poset-upper-triangular, but still computationally convenient.

We now outline the contents of the paper. In Section~\ref{sec:springer
  and circle action} we provide the necessary definitions and set some
notation.  In Section~\ref{sec:highest form} we define the
adjacent-pair matrix
and prove results concerning highest forms of linear
operators. As a simple application we derive the change-of-variable
matrix required to describe the circle subgroup of $T \subseteq
U(n,\C)$ acting on a translated Springer variety. Section~\ref{sec:S1
  action} contains our results on the bijection between permissible
fillings of a Young diagram and the $S^1$-fixed points of Hessenberg
varieties.  Section~\ref{sec:betti} is a mainly expository
section which recalls the terminology and definitions of poset pinball
and the dimension pair algorithm in \cite{HarTym10, BayHar10b}. In 
Sections~\ref{sec:examples} and~\ref{sec:pinball},
poset pinball for the case of $(n-2,2)$ Springer varieties is studied
in detail. The small-$n$ cases $n=4$ and $n=5$ are explicitly computed
and recorded in Section~\ref{sec:examples}.
The
main pinball result is in Section~\ref{sec:pinball}, where we prove that
the dimension pair algorithm 
yields 
a linearly independent set of classes in
$H^*_{S^1}(\mathcal{S}_{(n-2,2)})$ and hence a module basis. We close with some directions for future investigation in Section~\ref{sec:open questions}.

\bigskip
\noindent \textbf{Acknowledgements.} We thank Darius Bayegan, Erik
Insko, and Aba Mbirika for helpful conversations
and interest in this project. We are particularly grateful to Julianna 
Tymoczko for her ongoing support, as well as for finding (and suggesting ways to fix!) errors in an earlier draft of this paper and making many excellent suggestions for improving exposition. 
% Finally, we thank the support of an NSERC Undergraduate
% Research Student Award in the summer of 2010, which allowed the first
% author to engage in this undergraduate research project under the
% supervision of the second author. 

\section{Nilpotent Hessenberg 
  varieties and $S^1$-actions}\label{sec:springer and circle action}

We begin with the definition of the type $A$ nilpotent Hessenberg
varieties, of which the nilpotent Springer varieties are a special
case. We also recall the definition of 
a circle subgroup of the maximal torus $T$ of
$U(n,\C)$ which acts on any nilpotent Hessenberg 
variety. Since some of the discussion below applies to any nilpotent
Hessenberg variety, we present the general definition here.
We
work exclusively with type $A$ in this manuscript and hence 
omit it from our terminology below. 

Given a nilpotent operator $N: \C^n \to \C^n$, consider its Jordan
canonical form with weakly decreasing sizes of Jordan blocks.  Let
$\lambda_N$ denote the partition of $n$ with entries the sizes of the
Jordan blocks of $N$. Throughout this manuscript we identify
partitions of $n$ with the corresponding Young diagram.  For example,
if $N: \C^5 \to \C^5$ is the operator with corresponding matrix
\[
\begin{bmatrix} 0 & 1 & 0 & 0 & 0 \\ 0 & 0 & 1 & 0 & 0 \\ 0 & 0 &
  0 & 0 & 0 \\ 0 & 0 & 0 & 0 & 1 \\ 0 & 0 & 0 & 0 & 0 \end{bmatrix} 
\]
with respect to the standard basis of $\C^5$, then since the matrix has $2$
Jordan blocks of sizes $3$ and $2$ respectively, it has associated
Young diagram 
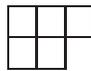
\[
{\def\lr#1{\multicolumn{1}{|@{\hspace{.6ex}}c@{\hspace{.6ex}}|}{\raisebox{-.3ex}{$#1$}}}
\raisebox{-.6ex}{$\begin{array}[b]{ccc}
\cline{1-1}\cline{2-2}\cline{3-3}
\lr{\phantom{1}}&\lr{\phantom{1}}&\lr{\phantom{1}}\\
\cline{1-1}\cline{2-2}\cline{3-3}
\lr{\phantom{1}}&\lr{\phantom{1}}\\
\cline{1-1}\cline{2-2}
\end{array}$}
}
\]
which in turn corresponds to the partition $\lambda_N = (3,2)$. 

A \textbf{Hessenberg function} is a function
$h: \{1,2,\ldots, n\} \to \{1,2,\ldots,
n\}$ satisfying $h(i) \geq i$ for all $1 \leq i \leq n$ and $h(i+1)
\geq h(i)$ for all $1 \leq i < n$. We frequently denote a Hessenberg
function by listing its values in sequence, $h = (h(1), h(2), \ldots,
h(n)=n)$. 

The \textbf{(nilpotent) Hessenberg variety} $\Hess(N,h)$ associated to $N$ and a Hessenberg
function $h$ is a subvariety of 
the flag variety 
$\mathcal{F}\ell ags(\C^n) \cong GL(n,\C)/B$.
Recall that
$\mathcal{F}\ell ags(\C^n)$ is the projective variety of nested
subspaces in $\C^n$, i.e. 
\[
\Flags(\C^n) = \{ V_{\bullet} = (V_i) : 0 \subseteq V_1 \subseteq V_2 \subseteq \cdots
\subseteq V_{n-1} \subseteq V_n = \mathbb{C}^n \textup{ such that } \dim_{\C}(V_i) = i\}.
\]  
Then $\Hess(N,h)$ 
is defined to be the following subvariety of 
$\mathcal{F}\ell ags(\C^n)$: 
\begin{equation}\label{eq:def-Hess}
\Hess(N,h) := \{ V_{\bullet}  \in \mathcal{F}\ell ags(\C^n) \;
\vert \;  N V_i \subseteq
V_{h(i)} \textup{ for all } 1 \leq i \leq n\}. 
\end{equation}
The (nilpotent) Springer varieties\footnote{In the
  literature they are also called \emph{Springer fibres} because they
  arise as fibres of the symplectic resolution $T^*\mathcal{F}\ell
  ags(\C^n) \to \mathcal{N}$ where $\mathcal{N}$ denotes the subspace
  of nilpotent matrices in $\mathfrak{gl}(n,\C)$, but we do not need
  or use this perspective here.} 
 are Hessenberg varieties for the
special case where the Hessenberg function is the identity function
$h(i)=i$ for all $1 \leq i \leq n$: 

\begin{definition}
  Let $N: \C^n \to \C^n$ be a nilpotent operator. The \textbf{Springer
    variety $\mathcal{S}_N$ associated to $N$} is defined as
\[
\mathcal{S}_N := \{ V_\bullet \in \Flags(\C^n) \; \vert \;  NV_i \subseteq V_i \textup{ for all }
1 \leq i \leq n\}.
\]
\end{definition}

For any \(g \in GL(n,\C),\) it is straightforward to see that the
Hessenberg variety $\Hess(gNg^{-1},h)$ for the conjugate $g N
g^{-1}$ of $N$ is homeomorphic (in fact, isomorphic as algebraic
varieties) to $\Hess(N,h)$, with explicit homeomorphism
given by translation by $g$, i.e.,
\begin{equation}\label{eq:SN to SgNg^{-1}}
\xymatrix @R=0.1in{
\Hess(N,h) \ar[r] & \Hess(g N g^{-1},h) \\
h \;  \ar @{|->}[r] & g h
}
\end{equation}
where \(h \in GL(n,\C)\) denotes a flag $[h] \in GL(n,\C)/B
\cong \Flags(\C^n)$.

There exists a circle action on any nilpotent Hessenberg variety. 
Recall first that the maximal torus $T$ of $U(n,\C)$,
identified with the diagonal subgroup of $U(n,\C)$, acts on
the flag variety $\Flags(\C^n)$. Consider the following circle
subgroup of $T$:
\begin{equation}\label{eq:def S^1 for Springer}
S^1 := \left\{ \left. \begin{bmatrix} t^n & 0 & \cdots & 0 \\ 0 & t^{n-1} &  &
      0 \\ 0 & 0 & \ddots & 0 \\ 0 & 0 &  & t \end{bmatrix}
  \; \right\rvert \;  t \in \C, \; \|t\| = 1 \right\}  \subseteq T
\subseteq U(n,\C).
\end{equation}
It is shown in \cite[Lemma 5.1]{HarTym10} that the
$S^1$ of~\eqref{eq:def S^1 for Springer} preserves the nilpotent
Hessenberg 
variety $\Hess(N,h) \subseteq \Flags(\C^n)$ when the nilpotent
operator $N$ has matrix in Jordan canonical form with respect
to the standard basis of $\C^n$. Moreover, the $S^1$-fixed points
$\Hess(N,h)^{S^1}$ are
isolated and are a subset of $\Flags(\C^n)^{T}$, the $T$-fixed
points of $\Flags(\C^n)$. Using the identification $\Flags(\C^n)^{T}
\cong S_n$ we henceforth think of $S^1$-fixed points of
$\Hess(N,h)$ as permutations in $S_n$.

\section{Adjacent-pair matrices and highest forms of nilpotent operators}\label{sec:highest form}

Suppose given a nilpotent matrix $N_0$ 
in standard Jordan
canonical form with weakly decreasing Jordan block sizes.  
We think of $N_0$ as a linear operator on $\C^n$ written with respect
to the standard basis of $\C^n$. As mentioned in
Section~\ref{sec:springer and circle action}, in addition to the
Hessenberg variety $\Hess(N_0,h)$ we may also consider the
translated Hessenberg varieties $\Hess(gN_0 g^{-1},h) = g \cdot
\Hess(N_0,h)$ for various $g \in GL(n,\C)$. For the purposes of
poset pinball (discussed in more detail in Section~\ref{sec:betti})  it turns out to be necessary to use conjugates $\sigma
N \sigma^{-1}$ where $\sigma$ is a permutation matrix and $\sigma N
\sigma^{-1}$ is in so-called \textbf{highest form} 
\cite[Definition 4.2]{Tym06}; this is because Tymoczko's construction
of a paving-by-affines of a Hessenberg variety $\Hess(N,h)$ \cite[Theorem
6.1]{Tym06} assumes that $N$ is in highest form. 
Motivated by this,
in this section we develop a theory which relates highest forms of $N_0$
with fillings of the corresponding Young diagram $\lambda =
\lambda_{N_0}$. First we introduce a
bijection $\phi_\lambda: \Fill(\lambda) \to S_n$ from
the set of \textbf{fillings} $\Fill(\lambda)$ of $\lambda$ and the
permutation group $S_n$. Secondly we associate to each filling $T$ of
$\lambda$ a matrix $N_T$ which we call the \textbf{adjacent-pair
  matrix} of $T$. The main results of this section are
Theorems~\ref{adjacency} and~\ref{theorem:N_T and
  algorithm}. Theorem~\ref{adjacency} 
observes that the adjacent-pair
matrix $N_T$ is precisely the conjugate $\sigma N_0 \sigma^{-1}$ where
$\sigma = \phi_\lambda(T)$ is the permutation corresponding to $T$
under the bijection $\phi_\lambda$. This gives a computationally easy
and explicit method for specifying the conjugates of $N_0$ by
permutation matrices. In Theorem~\ref{theorem:N_T and algorithm} we
then prove that $N_T = \sigma N_0 \sigma^{-1}$ is in highest form
precisely when $T$ arises from a certain simple algorithm which we
describe below. This yields a straightforward enumeration of all
permutation matrices $\sigma$ for which $\sigma N_0 \sigma^{-1}$ is in
highest form, and in particular in Corollary~\ref{corollary:number
  highest forms} we give a count of the number of conjugates $\sigma
N_0 \sigma^{-1}$ for $\sigma \in S_n$ which are in highest form. 

The discussion in this section has several motivations and consequences. Firstly, our
results (e.g. Corollary~\ref{corollary:tymoczko}) both make explicit and also generalize a procedure for
producing highest forms of linear operators which is sketched in
\cite[Section 4, text near Figure 4]{Tym06}. Secondly, our explicit
correspondence between certain fillings of $\lambda$ and highest forms
of $N_0$ allows us to easily determine the permutation $\sigma =
\phi_\lambda(T)$ (see e.g. Example~\ref{example:compute sigma}) and
thus make further explicit computations with $\sigma$. As a sample
such computation and for use in Section~\ref{sec:pinball}, at the end of this
section we give a concrete description in coordinates of the
conjugated circle $\sigma S^1 \sigma^{-1}$ which acts on
the Springer variety $\mathcal{S}_{\sigma N_0 \sigma^{-1}}$ for $N_0$
corresponding to $\lambda = (n-2,2)$, as well as a computation of the
associated Lie algebra projection $\Lie(T) \to \Lie(\sigma S^1
\sigma^{-1})$. Thus some of the results in this section are 
preliminary to the arguments in the sections below. Third, we believe
that the theory initiated here of highest forms in relation to
Springer varieties is of independent interest; we describe some open
questions motivated by this theory in 
Section~\ref{sec:open questions}.

We recall some definitions.

\begin{definition} \textbf{(\cite[Definition 4.1]{Tym06})}
Let $X$ be any $m \times n$ matrix . We call the entry $X_{ik}$ a 
\textbf{pivot} of $X$
if $X_{ik}$ is nonzero 
and if all entries below and to its left vanish, i.e.,
\(X_{ij} = 0\) if \(j < k\) and \(X_{jk} = 0\) if \(j > i.\) 
\end{definition}
Moreover, given $i$, define $r_i$ to be the row of $X_{r_i,i}$ if the
entry is a pivot, and $0$ otherwise.

\begin{example}
Let
\begin{equation*}
X = \left[
\begin{array}{cccc}
0&1&1&0 \\0&0&5&0 \\ 0&1&0&0 \\ 0&0&0&3
\end{array}
\right].
\end{equation*}
Then $r_1 = 0$, $r_2 = 3$, $r_3 = 2$, and $r_4 = 4$.
\end{example}

\begin{definition} \textbf{(\cite[Definition 4.2]{Tym06})}
An upper-triangular nilpotent $n \times n$ matrix is in
\textbf{highest 
form} if its pivots form a nondecreasing sequence, namely 
\(r_1 \leq r_2 \leq \cdots \leq r_n.\) 
\end{definition}

\begin{example}\label{example: yng(3,2,1)}
The nilpotent matrix
\[
N = \begin{bmatrix} 0 & 0 & 0 & 0 & 0 & 0 \\
0 & 0 & 0 & 1 & 0 & 0 \\
0 & 0 & 0 & 0 & 1 & 0 \\
0 & 0 & 0 & 0 & 0 & 0 \\
0 & 0 & 0 & 0 & 0 & 1 \\
0 & 0 & 0 & 0 & 0 & 0 
\end{bmatrix}
\]
is in highest form since $r_1 = r_2 = r_3 = 0$,
$r_4 = 2, r_5 = 3, r_6 = 5$. 
\end{example}

Recall that a
  \textbf{filling} of $\lambda$ by the alphabet
  $[n]:=\{1,2,\ldots, n\}$ is an injective placing of the
  integers $\{1,2,\ldots, n\}$ into the boxes of $\lambda$. Following tableaux notation we
  denote by $T$ a filling of $\lambda$ by $[n]$.  
We denote
  by $\mathcal{F}i\ell \ell(\lambda)$ the set of all fillings of
  $\lambda$ by $[n]$. For $\lambda$ a Young
  diagram with $n$ boxes, we have $\lvert \mathcal{F}i \ell \ell(\lambda)
  \rvert = n!$. In the theory below we use a particular bijective
  correspondence between $\Fill(\lambda)$ and $S_n$. We introduce the
  following terminology.

\begin{definition}\label{english}
Let $\lambda$ be a Young diagram. 
  Let $T$ be a filling of $\lambda$ with alphabet
  $[n]$ for some $n \in \N$.  By the \textbf{English reading} of
  $T$ we mean the reading of the entries of $T$ from
  left to right along rows, starting at the top row and proceeding in
  sequence to the bottom row. 
The word of $T$ obtained via the 
  English reading of $T$ is called the {\bf English word} of
  $T$. 
If $\lambda$ is a Young diagram with $n$ boxes then
  we define
\begin{equation}\label{eq:def phi_lambda} 
\phi_\lambda: \mathcal{F}i \ell \ell(\lambda) \to S_n
\end{equation}
where $\phi_\lambda(T)$ is the permutation whose one-line notation
is given by the English word of $T$. 
Finally, if $\lambda$ has $n$ boxes then the \textbf{English
  filling} of $\lambda$ is the filling $T$ such that
$\phi_\lambda(T)$ is the identity permutation in $S_n$. 
\end{definition}

For $\lambda$ a Young diagram with $n$ boxes, it is immediate from the
definition that $\phi_\lambda$ is a bijection from $\mathcal{F}i \ell
\ell(\lambda)$ to $S_n$.

\begin{example}
For 
\[
T = \myvcenter{\def\lr#1{\multicolumn{1}{|@{\hspace{.6ex}}c@{\hspace{.6ex}}|}{\raisebox{-.3ex}{$#1$}}}
\raisebox{-.6ex}{$\begin{array}[b]{cccc}
\cline{1-1}\cline{2-2}\cline{3-3}\cline{4-4}
\lr{1}&\lr{2}&\lr{3}&\lr{4}\\
\cline{1-1}\cline{2-2}\cline{3-3}\cline{4-4}
\lr{5}&\lr{6}\\
\cline{1-1}\cline{2-2}
\lr{7}\\
\cline{1-1}
\end{array}$}
}
\quad \textup{  and  } \quad
T' =
\myvcenter
{\def\lr#1{\multicolumn{1}{|@{\hspace{.6ex}}c@{\hspace{.6ex}}|}{\raisebox{-.3ex}{$#1$}}}
\raisebox{-.6ex}{$\begin{array}[b]{cccc}
\cline{1-1}\cline{2-2}\cline{3-3}\cline{4-4}
\lr{3}&\lr{5}&\lr{6}&\lr{7}\\
\cline{1-1}\cline{2-2}\cline{3-3}\cline{4-4}
\lr{2}&\lr{4}\\
\cline{1-1}\cline{2-2}
\lr{1}\\
\cline{1-1}
\end{array}$}
}
\]
we have that 
$\phi_\lambda(T)$ and $\phi_\lambda(T')$ are respectively the
permutations (in one-line notation) $1234567$
and $3567241$. Moreover $T$ is the English filling of $\lambda = (4,2,1)$. 
\end{example}

Next we introduce a different reading of fillings which 
appears in the theory of highest forms and
Hessenberg varieties developed by Tymoczko in \cite{Tym06} (but the
terminology we use is new). In particular, this reading plays a significant role in
our poset pinball methods in
Sections~\ref{sec:betti}-\ref{sec:pinball} (cf. in particular
Theorem~\ref{theorem:paving}).

\begin{definition}\label{our filling}
Let $\lambda$ be a Young diagram. 
Let $T$ be a filling of a Young diagram with alphabet
$[n]$ for some $n \in \N$. 
By the \textbf{rotated English
    reading} of $T$ we mean the reading of the entries of
  $T$ from the bottom to the top along columns, starting at
  the leftmost column and proceeding to the rightmost column. The word
  of $T$ 
  obtained via the rotated English reading is the \textbf{rotated
    English word} of $T$. 
Let $\lambda$ be a Young diagram with $n$ boxes.  The \textbf{rotated English
   filling} of $\lambda$ is the filling $T$ of $\lambda$
 with $[n]$ such that its rotated English reading is the identity
 permutation in $S_n$. 
\end{definition}

\begin{example}\label{example:rotated english}
  Suppose that $\lambda = \myvcenter
{\def\lr#1{\multicolumn{1}{|@{\hspace{.6ex}}c@{\hspace{.6ex}}|}{\raisebox{-.3ex}{$#1$}}}
\raisebox{-.6ex}{$\begin{array}[b]{cccc}
\cline{1-1}\cline{2-2}\cline{3-3}\cline{4-4}
\lr{\phantom{1}}&\lr{\phantom{1}}&\lr{\phantom{1}}&\lr{\phantom{1}}\\
\cline{1-1}\cline{2-2}\cline{3-3}\cline{4-4}
\lr{\phantom{1}}&\lr{\phantom{1}}\\
\cline{1-1}\cline{2-2}
\lr{\phantom{1}}\\
\cline{1-1}
\end{array}$}
}
$.  Then the
  rotated English filling of $\lambda$ is the filling
$\myvcenter
{\def\lr#1{\multicolumn{1}{|@{\hspace{.6ex}}c@{\hspace{.6ex}}|}{\raisebox{-.3ex}{$#1$}}}
\raisebox{-.6ex}{$\begin{array}[b]{cccc}
\cline{1-1}\cline{2-2}\cline{3-3}\cline{4-4}
\lr{3}&\lr{5}&\lr{6}&\lr{7}\\
\cline{1-1}\cline{2-2}\cline{3-3}\cline{4-4}
\lr{2}&\lr{4}\\
\cline{1-1}\cline{2-2}
\lr{1}\\
\cline{1-1}
\end{array}$}
}.
$
\end{example}

\begin{remark} 
Note that the rotated English filling is not the same thing as the
conjugate of the English filling of the conjugate Young diagram. 
For instance for the $\lambda$ in Example~\ref{example:rotated english} the 
conjugate of the English filling of the conjugate Young diagram $\tilde{\lambda}$ is 
$\myvcenter
{\def\lr#1{\multicolumn{1}{|@{\hspace{.6ex}}c@{\hspace{.6ex}}|}{\raisebox{-.3ex}{$#1$}}}
\raisebox{-.6ex}{$\begin{array}[b]{cccc}
\cline{1-1}\cline{2-2}\cline{3-3}\cline{4-4}
\lr{1}&\lr{4}&\lr{6}&\lr{7}\\
\cline{1-1}\cline{2-2}\cline{3-3}\cline{4-4}
\lr{2}&\lr{5}\\
\cline{1-1}\cline{2-2}
\lr{3}\\
\cline{1-1}
\end{array}$}
}
$
whereas the rotated English filling of $\lambda$ is 
$\myvcenter
{\def\lr#1{\multicolumn{1}{|@{\hspace{.6ex}}c@{\hspace{.6ex}}|}{\raisebox{-.3ex}{$#1$}}}
\raisebox{-.6ex}{$\begin{array}[b]{cccc}
\cline{1-1}\cline{2-2}\cline{3-3}\cline{4-4}
\lr{3}&\lr{5}&\lr{6}&\lr{7}\\
\cline{1-1}\cline{2-2}\cline{3-3}\cline{4-4}
\lr{2}&\lr{4}\\
\cline{1-1}\cline{2-2}
\lr{1}\\
\cline{1-1}
\end{array}$}.
}
$
\end{remark}

\begin{remark} 
In the next section we develop a more
general framework in which both Definition~\ref{english} and
Definition~\ref{our filling} are special cases, but we do not
need this perspective here. 
\end{remark}

Given a Young diagram with $n$ boxes and any filling $T$ of $\lambda$ by
$[n]$, we now construct a matrix we call
the \textbf{adjacent-pair matrix}. Our construction is a 
generalization of a procedure sketched by
Tymoczko in \cite[Section 4]{Tym06} (see in particular \cite[Figure 4]{Tym06}). 
We begin by defining adjacency in $\lambda$ and in a filling $T$.

\begin{definition}
Let $\lambda$ be a Young diagram.  We say that two
  boxes of $\lambda$ are \textbf{adjacent} if the two boxes are in the
  same row, and one box is directly to the left of the other.  That is,
  the two boxes are of the form $\myvcenter
{\def\lr#1{\multicolumn{1}{|@{\hspace{.6ex}}c@{\hspace{.6ex}}|}{\raisebox{-.3ex}{$#1$}}}
\raisebox{-.6ex}{$\begin{array}[b]{cc}
\cline{1-1}\cline{2-2}
\lr{\phantom{1}}&\lr{\phantom{1}}\\
\cline{1-1}\cline{2-2}
\end{array}$}
}$ within the Young
  diagram $\lambda$.  Similarly, given a filling $T$ of 
  $\lambda$, we say that two entries of $\mathcal{M}$ are {\bf
    adjacent}, or that they form an \textbf{adjacent pair}, if they occur in adjacent boxes.
\end{definition}

\begin{example}
For
\[
T = {\def\lr#1{\multicolumn{1}{|@{\hspace{.6ex}}c@{\hspace{.6ex}}|}{\raisebox{-.3ex}{$#1$}}}
\raisebox{-.6ex}{$\begin{array}[b]{ccc}
\cline{1-1}\cline{2-2}\cline{3-3}
\lr{1}&\lr{2}&\lr{3}\\
\cline{1-1}\cline{2-2}\cline{3-3}
\lr{4}&\lr{5}\\
\cline{1-1}\cline{2-2}
\lr{6}\\
\cline{1-1}
\end{array}$}
}
\]
the pairs $\{1,2\}$, $\{2, 3\}$, and $\{4,5\}$ are the adjacent pairs of
entries of $T$.
\end{example}

\begin{definition} 
Let $\lambda$
  be a Young diagram with $n$ boxes and $T$ a
  filling of $\lambda$ with entries from $[n]$.  Then we define the
  \textbf{adjacent-pair matrix} corresponding to $T$, denoted $N_T$, to be the 
  matrix $N_T = (a_{ij})_{1 \leq i, j \leq n}$ such that its $(i,j)$-th entry
  is given by 
\[
a_{ij} := 
\begin{cases}
1& \text{if $i$ and $j$ are adjacent in $T$ and $i$ is left of $j$},\\ 0& \text{otherwise}.
\end{cases}
\]
\end{definition}

\begin{example}
Suppose that $\lambda = \myvcenter
{\def\lr#1{\multicolumn{1}{|@{\hspace{.6ex}}c@{\hspace{.6ex}}|}{\raisebox{-.3ex}{$#1$}}}
\raisebox{-.6ex}{$\begin{array}[b]{ccc}
\cline{1-1}\cline{2-2}\cline{3-3}
\lr{\phantom{1}}&\lr{\phantom{1}}&\lr{\phantom{1}}\\
\cline{1-1}\cline{2-2}\cline{3-3}
\lr{\phantom{1}}&\lr{\phantom{1}}\\
\cline{1-1}\cline{2-2}
\lr{\phantom{1}}\\
\cline{1-1}
\end{array}$}
}$ and that
$\myvcenter T  = 
{\def\lr#1{\multicolumn{1}{|@{\hspace{.6ex}}c@{\hspace{.6ex}}|}{\raisebox{-.3ex}{$#1$}}}
\raisebox{-.6ex}{$\begin{array}[b]{ccc}
\cline{1-1}\cline{2-2}\cline{3-3}
\lr{3}&\lr{2}&\lr{4}\\
\cline{1-1}\cline{2-2}\cline{3-3}
\lr{1}&\lr{5}\\
\cline{1-1}\cline{2-2}
\lr{6}\\
\cline{1-1}
\end{array}$}
}.
$
Then
\[
N_T = \left[
\begin{array}{cccccc}
0&0&0&0&1&0 \\ 0&0&0&1&0&0 \\ 0&1&0&0&0&0 \\ 0&0&0&0&0&0 \\  0&0&0&0&0&0 \\  0&0&0&0&0&0\end{array}
\right].
\]
\end{example}

\begin{remark}\label{remark:english filling}
The adjacent-pair matrix $N_T$ corresponding to the English filling of
$\lambda$ is the nilpotent matrix in Jordan canonical form
corresponding to $\lambda$. 
For example if 
$\myvcenter T = 
{\def\lr#1{\multicolumn{1}{|@{\hspace{.6ex}}c@{\hspace{.6ex}}|}{\raisebox{-.3ex}{$#1$}}}
\raisebox{-.6ex}{$\begin{array}[b]{ccc}
\cline{1-1}\cline{2-2}\cline{3-3}
\lr{1}&\lr{2}&\lr{3}\\
\cline{1-1}\cline{2-2}\cline{3-3}
\lr{4}&\lr{5}\\
\cline{1-1}\cline{2-2}
\lr{6}\\
\cline{1-1}
\end{array}$}
}
$
then
\[
N_T = \begin{bmatrix} 0 & 1 & 0 & 0 & 0 & 0 \\
0 & 0 & 1 & 0 & 0 & 0 \\
0 & 0 & 0 & 0 & 0 & 0 \\
0 & 0 & 0 & 0 & 1 & 0 \\
0 & 0 & 0 & 0 & 0 & 0 \\
0 & 0 & 0 & 0 & 0 & 0 
\end{bmatrix}.
\]
\end{remark}

The following is a basic computation which relates adjacent-pair
matrices to highest forms. 
Given a permutation $\sigma \in S_n$ by slight abuse
of notation we denote also by $\sigma$ its $n \times n$ permutation
matrix with respect to the standard basis of $\C^n$, i.e., the matrix
with $i$-th column equal to the standard basis vector
$e_{\sigma(i)}$.

\begin{theorem}\label{adjacency}
  Let $N$ be an $n \times n$ nilpotent matrix in Jordan canonical form with
  weakly decreasing sizes of Jordan blocks. Let $\lambda_N$ be the
  corresponding Young diagram and let
 $L_N: \C^n \rightarrow \C^n$ be the linear operator
with matrix 
  $N$ with respect to the standard basis 
  of $\C^n$. 
  Let $T$ be a filling on $\lambda_N$ with alphabet $[n]$ and
  $\sigma:=\phi_\lambda(T) \in S_n$ 
  the permutation given by the English
  word of $T$.
  Then the adjacent-pair matrix $N_T$ corresponding to $T$ is equal to 
the conjugate
  $\sigma N\sigma^{-1}$, i.e., $N_T$ is the matrix of $L_N$ with respect
  to the basis $\{e_{\sigma^{-1}(1)}, \ldots, e_{\sigma^{-1}(n)}\}$.
\end{theorem}

\begin{proof}
By the 
definition of the adjacent-pair matrix, 
given a Young diagram $\lambda_N$ with $n$ boxes and $\ell$
rows, $N_T$ contains a $1$ in $n-\ell$ entries, and all
other entries are $0$.  Similarly, an $n\times n$ nilpotent
matrix $N$, with corresponding Young diagram $\lambda_N$,
contains a $1$ in $n-\ell$ entries and $0$'s elsewhere, and so does any conjugate 
$\sigma N \sigma^{-1}$ for $\sigma$ a permutation (matrix).

Now let $a_{ij}$ denote the $(i,j)$-th entry of $M$.  The preceding
discussion implies that in order to prove the proposition it suffices
to check that if $a_{ij}=1$ for some $i$ and $j$, then the matrix of
$L$ with respect to the basis $\{e_{\sigma^{-1}(1)}, \ldots,
e_{\sigma^{-1}(n)}\}$ also contains a $1$ at the $(i,j)$-th entry.
Suppose that $a_{ij}=1$. By construction this means that $i$ and $j$
are adjacent in the filling $T$ with $i$ to the left of
$j$. Hence by definition of the English reading the one-line notation of $\sigma$ is of the form $\sigma =
\cdots i \, j \cdots$. Suppose that the $i$ occurs at the $\ell$-th
spot of the one-line notation, so $\sigma(\ell)=i, \sigma(\ell+1) =
j$. Then $\sigma^{-1}(i) = \ell, \sigma^{-1}(j) = \ell+1$. Since $i$
and $j$ are adjacent, 
we also know that 
$L(e_{\ell+1}) = e_\ell$ (cf. Remark~\ref{remark:english filling})  or equivalently $L(e_{\sigma^{-1}(j)}) =
e_{\sigma^{-1}(i)}$. This implies that the matrix of $L$ written with
respect to the basis $\{e_{\sigma^{-1}(1)}, \ldots,
e_{\sigma^{-1}(n)}\}$ has a $1$ in the $(i,j)$-th entry, as desired. 

\end{proof}

We now wish to determine the set of fillings $T$ such that the
adjacent-pair matrix $N_T$ is in highest form. 
Throughout this discussion we use the following assumptions and notation. Let $\lambda$
be a Young diagram with $n$ boxes, $\ell$ rows, $k$ rows of distinct
length, and $r$ columns. If $\lambda_1 > \lambda_2 > \cdots
> \lambda_k$ are the distinct row lengths of $\lambda$ we let $d_i$
for 
$1 \leq i \leq k$ denote the number of rows of $\lambda$ with length
$\lambda_i$. Thus the row lengths of $\lambda$ are 
\[
(\underbrace{\lambda_1, \lambda_1, \ldots, \lambda_1}_{d_1}, 
\underbrace{\lambda_2, \lambda_2, \ldots, \lambda_2}_{d_2}, 
\cdots, 
\underbrace{\lambda_k, \lambda_k, \ldots, \lambda_k}_{d_k})
\]
with $\sum_{i=1}^k d_i = \ell$. We also let $(\mu_1 \geq \mu_2 \geq
\cdots \geq \mu_r)$ denote the column lengths of $\lambda$. Note
$\mu_1 = \ell$. 

We begin with some observations about the pivots of an adjacent-pair
matrix $N_T$. 

\begin{lemma}\label{lemma:pivots}
Let $\lambda$ be a Young diagram with $n$ boxes and $T$ a filling of
$\lambda$ by $[n]$. Let 
$N_T$ be the adjacent-pair matrix
  of $T$. Then each matrix entry in $N_T$ which is equal to $1$ is a pivot of $N_T$. 

\end{lemma}

\begin{proof}
By definition of the adjacent-pair matrix,
its non-zero entries are in one-to-one
  correspondence with the distinct adjacent pairs \(
\begin{array}{|c|c|}
    \cline{1-2} i & j \\ \cline{1-2} \end{array}
\) which appear in $T$. Each box of $\lambda$ which is not in the
leftmost (i.e. first) column of $\lambda$ is the right hand box of
precisely one such adjacent pair \(\begin{array}{|c|c|}
    \cline{1-2} \phantom{i} & \phantom{j} \\ \cline{1-2} \end{array}
\) of boxes. Hence each such accounts for precisely one entry of $N_T$
equal to $1$.

By definition of fillings, each entry in $T$ appears only once. In
particular this means that for any $i \in [n]$, the index $i$ appears
at most once as either the right hand box \(\begin{array}{|c|c|}
    \cline{1-2} j & i \\ \cline{1-2} \end{array}
\) or the left hand box \(\begin{array}{|c|c|}
    \cline{1-2} i & j \\ \cline{1-2} \end{array}
\) in an adjacent pair in $T$. Thus by definition of the adjacent-pair
matrix there exists at most one entry equal to $1$ in each row and
each column of $N_T$. Since all other entries are equal to $0$, 
this in turn implies that each $1$ that appears in $N_T$
is in fact a pivot. 
\end{proof}

\begin{lemma}\label{lemma:N_T equivalence}
Let $\lambda$ be a Young diagram with $n$ boxes and $T$ a filling of
$\lambda$ by $[n]$. Let 
$N_T$ be the adjacent-pair matrix
  of $T$. 
Then $N_T$ is in highest form if and only if $T$
satisfies the following conditions:
\begin{enumerate}
\item[(a)] the leftmost column of $\lambda$ is filled with the
  integers $\{1,2,\ldots,\mu_1=\ell\}$, and 
\item[(b)] if \(\begin{array}{|c|c|}
    \cline{1-2} i_1 & j_1 \\ \cline{1-2} \end{array}
\) and \(\begin{array}{|c|c|}
    \cline{1-2} i_2 & j_2 \\ \cline{1-2} \end{array}
\) both appear as adjacent pairs in $T$ then 
\[
i_1 < i_2 \textup{ if and only if } j_1 < j_2.
\]
\end{enumerate}
\end{lemma}

\begin{proof}
First suppose $N_T$ is in highest form. By Lemma~\ref{lemma:pivots} if 
  \(\begin{array}{|c|c|}
    \cline{1-2} i & j \\ \cline{1-2} \end{array}
\) appears as an adjacent pair in $T$ then $r_j = i > 0$. For $j \in
[n]$, the index $j$ does not appear in the right hand box of any
adjacent pair in $T$ (so the $j$-th column of $N_T$ is identically
$0$) precisely when $j$ appears in the leftmost (i.e. first) column of
$\lambda$. In this case, by definition of pivots, $r_j = 0$. Since $N_T$
is in highest form we must have $r_1 \leq r_2 \leq \cdots \leq r_n$
and in particular any $r_j$ with $r_j=0$ must occur before any $r_j$
with $r_j>0$. We conclude that $j$ is in the leftmost column of
$\lambda$ precisely when $1 \leq j \leq \mu_1 =\ell$. This proves
(a). Now suppose \(\begin{array}{|c|c|}
    \cline{1-2} i_1 & j_1 \\ \cline{1-2} \end{array}
\) and \(\begin{array}{|c|c|}
    \cline{1-2} i_2 & j_2 \\ \cline{1-2} \end{array}
\) both appear as adjacent pairs in $T$. Then again from
Lemma~\ref{lemma:pivots} we know $r_{j_1} = i_1, r_{j_2} = i_2$. If
$N_T$ is in highest form then the pivots must be increasing so
$j_1<j_2$ if and only if $i_1 < i_2$. This proves (b). If
$T$ satisfies conditions (a) and (b) then reversing this reasoning shows
that $N_T$ must be in highest form. 
\end{proof}

We now describe an algorithm which produces a filling $T$ of $\lambda$
which satisfies certain conditions, starting from the data of a
filling of the leftmost column of $\lambda$. As we show in
Theorem~\ref{theorem:N_T and algorithm} below, the algorithm gives
an explicit method for producing precisely those fillings $T$ for
which the corresponding $N_T$ are in highest form. We follow notation
established above. 
\begin{equation}\label{eq:def algorithm} 
\begin{minipage}{0.7\linewidth}
  \begin{enumerate}
  \item Fix an arbitrary filling of the leftmost (i.e. first) column
    of $\lambda$ with the alphabet $[\mu_1]$. This filling specifies a
    linear ordering of the rows of $\lambda$. 
 \item For the $s$-th column of $\lambda$ for $2 \leq s
    \leq r$, place the $\mu_s$ integers $\{(\sum_{t=1}^{s-1} \mu_t)+1,
    \ldots, \sum_{t=1}^s \mu_t\}$ in the $\mu_s$ boxes of the $s$-th
    column in the linear order specified by step (1). 
  \end{enumerate}
\end{minipage}
\end{equation}
Note that, by definition of this algorithm, the filling of the leftmost column
completely specifies the rest of the filling. 

\begin{example}
If the Young diagram $\lambda$ and the initial filling of its
leftmost column are 
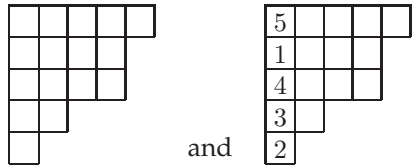
\[
{\def\lr#1{\multicolumn{1}{|@{\hspace{.6ex}}c@{\hspace{.6ex}}|}{\raisebox{-.3ex}{$#1$}}}
\raisebox{-.6ex}{$\begin{array}[b]{ccccc}
\cline{1-1}\cline{2-2}\cline{3-3}\cline{4-4}\cline{5-5}
\lr{\phantom{1}}&\lr{\phantom{1}}&\lr{\phantom{1}}&\lr{\phantom{1}}&\lr{\phantom{1}}\\
\cline{1-1}\cline{2-2}\cline{3-3}\cline{4-4}\cline{5-5}
\lr{\phantom{1}}&\lr{\phantom{1}}&\lr{\phantom{1}}&\lr{\phantom{1}}\\
\cline{1-1}\cline{2-2}\cline{3-3}\cline{4-4}
\lr{\phantom{1}}&\lr{\phantom{1}}&\lr{\phantom{1}}&\lr{\phantom{1}}\\
\cline{1-1}\cline{2-2}\cline{3-3}\cline{4-4}
\lr{\phantom{1}}&\lr{\phantom{1}}\\
\cline{1-1}\cline{2-2}
\lr{\phantom{1}}\\
\cline{1-1}
\end{array}$}
}
\quad \textup{ and } \quad 
{\def\lr#1{\multicolumn{1}{|@{\hspace{.6ex}}c@{\hspace{.6ex}}|}{\raisebox{-.3ex}{$#1$}}}
\raisebox{-.6ex}{$\begin{array}[b]{ccccc}
\cline{1-1}\cline{2-2}\cline{3-3}\cline{4-4}\cline{5-5}
\lr{5}&\lr{\phantom{1}}&\lr{\phantom{1}}&\lr{\phantom{1}}&\lr{\phantom{1}}\\
\cline{1-1}\cline{2-2}\cline{3-3}\cline{4-4}\cline{5-5}
\lr{1}&\lr{\phantom{1}}&\lr{\phantom{1}}&\lr{\phantom{1}}\\
\cline{1-1}\cline{2-2}\cline{3-3}\cline{4-4}
\lr{4}&\lr{\phantom{1}}&\lr{\phantom{1}}&\lr{\phantom{1}}\\
\cline{1-1}\cline{2-2}\cline{3-3}\cline{4-4}
\lr{3}&\lr{\phantom{1}}\\
\cline{1-1}\cline{2-2}
\lr{2}\\
\cline{1-1}
\end{array}$}
}
\]
then the algorithm~\eqref{eq:def algorithm} determines the rest of the
filling to be 
\[
{\def\lr#1{\multicolumn{1}{|@{\hspace{.6ex}}c@{\hspace{.6ex}}|}{\raisebox{-.3ex}{$#1$}}}
\raisebox{-.6ex}{$\begin{array}[b]{ccccc}
\cline{1-1}\cline{2-2}\cline{3-3}\cline{4-4}\cline{5-5}
\lr{5}&\lr{9}&\lr{12}&\lr{15}&\lr{16}\\
\cline{1-1}\cline{2-2}\cline{3-3}\cline{4-4}\cline{5-5}
\lr{1}&\lr{6}&\lr{10}&\lr{13}\\
\cline{1-1}\cline{2-2}\cline{3-3}\cline{4-4}
\lr{4}&\lr{8}&\lr{11}&\lr{14}\\
\cline{1-1}\cline{2-2}\cline{3-3}\cline{4-4}
\lr{3}&\lr{7}\\
\cline{1-1}\cline{2-2}
\lr{2}\\
\cline{1-1}
\end{array}$}
}.
\]
\end{example}

\begin{remark}\label{remark:T rotated english}
  Suppose the filling of the leftmost column of $\lambda$ is given by
  placing the integer $i$, for $1 \leq i \leq \mu_1$, in the $i$-th
  box from the bottom. Then the filling of $\lambda$ obtained by
  applying the algorithm~\eqref{eq:def algorithm} is precisely the
  rotated English filling of Definition~\ref{our filling}. 
\end{remark}

We now prove that the fillings $T$ for which $N_T$ is in highest form
are precisely those produced from the algorithm~\eqref{eq:def algorithm}. 

\begin{theorem}\label{theorem:N_T and algorithm}
Let $\lambda$ be a Young diagram with $n$ boxes and $T$ a filling of
$\lambda$ by $[n]$.
Then the adjacent-pair matrix
  $N_T$ is in highest form if and only if the 
  algorithm~\eqref{eq:def algorithm} applied to the filling of the 
  leftmost column of $T$ produces the filling $T$. 
\end{theorem}

To prove the proposition we use the following lemma. We follow notation established above. 
\begin{lemma}\label{lemma:column fillings}
Let $\lambda$ be a Young diagram with $n$ boxes and $T$ a filling of
$\lambda$ by $[n]$. 
Suppose $T$ satisfies the conditions (a)
  and (b) of Lemma~\ref{lemma:N_T equivalence}. Then the $s$-th column
  of $\lambda$ for $1 \leq s \leq r$ contains precisely the integers
  $\{(\sum_{t=1}^{s-1} \mu_t)+1, \ldots, \sum_{t=1}^s \mu_t\}$. 
\end{lemma}

\begin{proof}
  We argue by induction. Condition (a) already implies the
  leftmost column is filled with $[\mu_1]$, which proves the base case
  $s=1$. Now suppose the first $s$ columns contain precisely the
  integers $\{1,2,\ldots, \sum_{t=1}^s \mu_t\}$. Suppose for a
  contradiction that some element $u$ in $\{(\sum_{t=1}^s \mu_t)+1,
  \ldots, \sum_{t=1}^{s+1}\mu_t\}$ appears in the $v$-th column for
  some $v>s+1$. Since there are precisely $\mu_{s+1}$ boxes in the
  $(s+1)$th column, this in turn implies that there must exist some $u'
  > \sum_{t=1}^{s+1}\mu_t$ that appears in the $(s+1)$th column. Thus
  there exist adjacent pairs \(\begin{array}{|c|c|}
    \cline{1-2} i_1 & u' \\ \cline{1-2} \end{array}
\) and \(\begin{array}{|c|c|}
    \cline{1-2} i_2 & u \\ \cline{1-2} \end{array}
\) with the properties that 
\begin{itemize}
\item $i_1 \leq \sum_{t=1}^s \mu_t$ and
\item $i_2 > \sum_{t=1}^s \mu_t$
\end{itemize}
 since $u'$ appears in the
  $(s+1)$th column and all entries in the $s$th column are less than
  or equal to $\sum_{t=1}^s \mu_t$ by assumption, and since $u$ appears in a column
  strictly to the right of the $(s+1)$th column. 
Thus $i_1 < i_2$ but $u'>u$, which contradicts condition (b). The
result follows. 
\end{proof}

\begin{proof}[Proof of Theorem~\ref{theorem:N_T and algorithm}]
 By Lemma~\ref{lemma:N_T equivalence} it suffices to prove that a
  filling $T$ satisfies conditions (a) and (b) of Lemma~\ref{lemma:N_T
    equivalence} if and only if it arises from~\eqref{eq:def
    algorithm}. So suppose $T$ satisfies Lemma~\ref{lemma:N_T
    equivalence}(a) and (b). From Lemma~\ref{lemma:column fillings} we
  already know that the set of entries in each column agrees with that
  specified by~\eqref{eq:def algorithm}, so it remains to show that
  the ordering of the entries also agrees, i.e. that the entries of
  the $s$-th column for $2 \leq s \leq r$ respects the linear order
  imposed on the rows by the filling of the leftmost column. We argue
  by induction. Suppose $s=2$. Then the entries of the $2$nd column
  respect the ordering in the $1$st column precisely when the following
  holds: if \(\begin{array}{|c|c|}
    \cline{1-2} i_1 & j_1 \\ \cline{1-2} \end{array}
\) and \(\begin{array}{|c|c|}
    \cline{1-2} i_2 & j_2 \\ \cline{1-2} \end{array}
\) are two adjacent pairs with $j_1, j_2$ in the $2$nd column of
$\lambda$ then $i_1<i_2$ if and only if $j_1 < j_2$. But this follows from
condition (b). Moreover if this
condition holds it follows that the linear ordering of the boxes in
the $2$nd column given by its filling by
$\{\mu_1+1,\ldots,\mu_1+\mu_2\}$ agrees with that induced by the
linear ordering of the rows of $\lambda$ corresponding to the filling of the $1$st
column. Assuming the first $s$ columns are obtained by~\eqref{eq:def
  algorithm}, the same argument as above shows that the $(s+1)$st
column must also be filled according to~\eqref{eq:def algorithm}, as
desired.

Conversely, suppose $T$ is obtained from~\eqref{eq:def algorithm}.  By
construction $T$ satisfies condition (a). Now suppose
\(\begin{array}{|c|c|} \cline{1-2} i_1 & j_1 \\
  \cline{1-2} \end{array} \) and \(\begin{array}{|c|c|} \cline{1-2}
  i_2 & j_2 \\ \cline{1-2} \end{array} \) are two adjacent pairs
appearing in $T$. 
We consider cases. Suppose $i_1$ and $i_2$ appear in
the $s$th and $s'$th columns of $T$. Without loss of generality we may
assume $s<s'$. Then $i_1 \leq \sum_{t=1}^s \mu_t$ while $i_2 \geq
(\sum_{t=1}^{s'-1} \mu_t)+1 \geq i_1$. Thus we wish to show $j_1 <
j_2$. This follows because the adjacency with $i_1$ and $i_2$
respectively implies that $j_1$ is in the $(s+1)$th column and $j_2$
is in the $(s'+1)$th column. Since $s+1<s'+1$ an argument similar to
that above implies $j_1 < j_2$ as desired. On the other hand suppose
$i_1$ and $i_2$ appear in the same column, say the $s$th. Then $j_1$
and $j_2$ appear in the $s+1$th column. 
Suppose further that $i_1$
appears in the $r_{i_1}$th row and $i_2$ appears in the $r_{i_2}$th
row. If $i_1<i_2$ then by definition of the algorithm~\eqref{eq:def algorithm}
the entry in the $r_{i_1}$th row of the first column is
less than that in the $r_{i_2}$th row, which in turn implies $j_1<j_2$.
Similarly $j_1<j_2$ implies $i_1<i_2$. This concludes the
proof.

\end{proof}

The following, asserted in
\cite[Section
4, see e.g. Figure 4]{Tym06}, is now a straightforward consequence.

\begin{corollary}\label{corollary:tymoczko}
  Let $\lambda$ be a Young diagram with $n$ boxes and $T_{RE}$ be the
  rotated English filling of $\lambda$. 
  Let $\sigma := \phi_\lambda(T_{RE})$ be the permutation given by the
  English reading of $T_{RE}$. Then $N_{T_{RE}} = \sigma N
  \sigma^{-1}$ is in highest form. 
\end{corollary}

\begin{proof} 
Immediate from
Theorems~\ref{adjacency} and~\ref{theorem:N_T and algorithm} and Remark~\ref{remark:T
  rotated english}. 
\end{proof}

We have just seen that each filling $T$ obtained from~\eqref{eq:def
  algorithm} yields a conjugate $N_T = \sigma N \sigma^{-1}$ of $N$ in
highest form. Since a filling given in~\eqref{eq:def algorithm} is
specified by the filling of its leftmost column, there are $\mu_1 ! =
\lvert S_{\mu_1} \rvert$ many such fillings. However, different
such fillings $T$ and $T'$ may yield the same adjacent-pair matrix
$N_T = N_{T'}$. The next lemma makes this precise, for the purpose of which we use
the following terminology. We say a filling $T'$ is obtained from $T$
by a \textbf{row swap} if the entries of $2$ equal-length rows of $T$
have been interchanged; more precisely, if both the $a$th row and the
$b$th row of $\lambda$ have $d$ boxes and entries 
\(\begin{array}{|c|c|c|}
    \cline{1-3} a_1 & \cdots & a_d \\ \cline{1-3} \end{array}
\)
and
\(\begin{array}{|c|c|c|}
    \cline{1-3} b_1 & \cdots & b_d \\ \cline{1-3} \end{array}
\)
respectively, then $T'$ is obtained from $T$ by swapping the $a$th
row and $b$th row if $T'$ contains the same entries as in $T$ in all
other rows, and the $a$th row of $T'$ has entries 
\(\begin{array}{|c|c|c|}
    \cline{1-3} b_1 & \cdots & b_d \\ \cline{1-3} \end{array}
\)
and the $b$th row has entries
\(\begin{array}{|c|c|c|}
    \cline{1-3} a_1 & \cdots & a_d \\ \cline{1-3} \end{array}.
\)

\begin{lemma}\label{lemma:row swap}
Let $\lambda$ be a Young diagram with $n$ boxes and let $T$ and $T'$ be fillings of $\lambda$ obtained from~\eqref{eq:def
  algorithm}. Then $N_T = N_{T'}$ if and only if $T'$ is obtained from
$T$ by a sequence of row swaps.   
\end{lemma}

\begin{proof}
  If $T$ and $T'$ differ only by a sequence of row swaps, then $T$ and
  $T'$ have precisely the same sets of adjacent pairs. Thus from the
  definition of the adjacent-pair matrix it follows that $N_T =
  N_{T'}$. Now suppose $T$ and $T'$ differ by more than a sequence of
  row swaps. Since both $T$ and $T'$ are obtained from~\eqref{eq:def
    algorithm}, this means that there exists an element $s \in
  [\mu_1]$ which appears in $T$ in a row of length $d$ and 
  appears in $T'$ in a row of length $d'$, with $d \neq d'$. Without
  loss of generality we assume $d' > d$. We wish to show that $N_T
  \neq N_{T'}$. For this it suffices to show that there exists some adjacent
  pair 
\(\begin{array}{|c|c|}
    \cline{1-2} i & j \\ \cline{1-2} \end{array}
\) which occurs in $T$ but not in $T'$, or vice versa. Consider the entries
in the row of $T$ and $T'$ which contain $s$. By assumption these are
of the form 
\(\begin{array}{|c|c|c|c|}
    \cline{1-4} a_1=s & a_2 & \cdots & a_d \\ \cline{1-4} \end{array}
\)
and
\(\begin{array}{|c|c|c|c|c|c|}
    \cline{1-6} a'_1=s & a'_2 & \cdots & a'_d  & \cdots & a'_{d'}\\ \cline{1-6} \end{array}
\)
respectively where $d'>d$. We take cases. Suppose there exists an
index $1 < i \leq d$ for which $a_i \neq a'_i$. Then in particular
there exists a minimal such, denote it $i$. Then there is an adjacent
pair \(\begin{array}{|c|c|}
    \cline{1-2} a_{i-1} & a_i \\ \cline{1-2} \end{array}
\) in $T$ and a pair 
\(\begin{array}{|c|c|}
    \cline{1-2} a'_{i-1}  & a'_i \\ \cline{1-2} \end{array}
\) in $T'$ where $a_{i-1} = a'_{i-1}$ but $a_i \neq a'_i$, so $N_T
\neq N_{T'}$. Now suppose $a_i = a'_i$ for all $1 \leq i \leq d$. In
particular $a_d = a'_d$. Then \(\begin{array}{|c|c|}
    \cline{1-2} a'_{d} = a_d  & a'_{d+1} \\ \cline{1-2} \end{array}
\) is an adjacent pair in $T'$ which does not occur in $T$. Hence $N_T
\neq N_{T'}$ also in this case. The result follows. 
\end{proof}

The following is now straightforward. Recall $\mu_1 = \ell$ is the
total number of rows of $\lambda$ and $d_1, \ldots, d_k$ are the numbers of
rows of $\lambda$ of length $\lambda_1, \ldots, \lambda_k$
respectively.

\begin{corollary}\label{corollary:number highest forms}
  There exist precisely
\[
\frac{\ell!}{d_1! d_2! \cdots d_k!}
\]
highest forms of $N$ obtained as $\sigma N \sigma^{-1}$ for a
permutation matrix $\sigma \in S_n$. 
\end{corollary}

\begin{proof}
  There are $\mu_1! = \ell!$ fillings $T$ 
  arising from the algorithm~\eqref{eq:def algorithm}. From
  Lemma~\ref{lemma:row swap} we know that the matrices $N_T$ do not
  change precisely when the entries in the first column contained in equal-length
  rows are permuted. The $d_i$ count the numbers of equal-length rows
  so the result follows. 
\end{proof}

Our constructions allow us to do explicit 
computations.
For instance, given the discussion above it is straightforward to list the permutation
matrices $\sigma$ for which the associated conjugate $\sigma N
\sigma^{-1}$ is in highest form. 
For instance, let $T_{RE}$ be the rotated
English filling of $\lambda$. It follows from the results above
that the permutation $\sigma$ for which $\sigma N \sigma^{-1}$ is the
choice of highest form of $N$ used in \cite[Section
4]{Tym06} is precisely $\sigma := \phi_\lambda(T)$.

\begin{example}\label{example:compute sigma}
Suppose the Young diagram is 
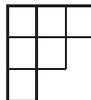
\[
{\def\lr#1{\multicolumn{1}{|@{\hspace{.6ex}}c@{\hspace{.6ex}}|}{\raisebox{-.3ex}{$#1$}}}
\raisebox{-.6ex}{$\begin{array}[b]{ccc}
\cline{1-1}\cline{2-2}\cline{3-3}
\lr{\phantom{1}}&\lr{\phantom{1}}&\lr{\phantom{1}}\\
\cline{1-1}\cline{2-2}\cline{3-3}
\lr{\phantom{1}}&\lr{\phantom{1}}\\
\cline{1-1}\cline{2-2}
\lr{\phantom{1}}\\
\cline{1-1}
\end{array}$}
}
\]
corresponding to the nilpotent matrix $N$ in Example~\ref{example: yng(3,2,1)}.
Then the rotated English filling is 
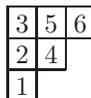
\[
{\def\lr#1{\multicolumn{1}{|@{\hspace{.6ex}}c@{\hspace{.6ex}}|}{\raisebox{-.3ex}{$#1$}}}
\raisebox{-.6ex}{$\begin{array}[b]{ccc}
\cline{1-1}\cline{2-2}\cline{3-3}
\lr{3}&\lr{5}&\lr{6}\\
\cline{1-1}\cline{2-2}\cline{3-3}
\lr{2}&\lr{4}\\
\cline{1-1}\cline{2-2}
\lr{1}\\
\cline{1-1}
\end{array}$}
}
\]
and the permutation $\sigma$ 
such that $\sigma N \sigma^{-1}$ is in the highest form used in
\cite[Section 4]{Tym06} is $\sigma = 356241$. 
\end{example}

As another application of our discussion and for use in Section~\ref{sec:pinball}
we close this section with a brief discussion about the circle action
on Hessenberg varieties defined in~\eqref{eq:def S^1 for Springer}. 
Consider the translated Hessenberg variety $\Hess(\sigma N
  \sigma^{-1}, h)$ where $N$ is in standard Jordan canonical form and
$\sigma$ is a permutation matrix. In this case the circle subgroup
of~\eqref{eq:def S^1 for Springer} does not necessarily act on
$\Hess(\sigma N \sigma^{-1},h)$. Instead we consider the
conjugated circle subgroup $\sigma S^1 \sigma^{-1}$ of $T$, which is
easily seen to preserve $\Hess(\sigma N \sigma^{-1},h)$. Here and
below we consider each such Springer variety to be equipped with this
conjugated circle group action, which by slight abuse of notation we
sometimes denote also by $S^1$ (instead of $\sigma S^1 \sigma^{-1}$). It is
immediate that the fixed points $\Hess(\sigma N
  \sigma^{-1},h)^{S^1}$ under the $S^1$-action are isolated and are a
subset of $S_n \cong \Flags(\C^n)^{T}$; indeed, under the
homeomorphism~\eqref{eq:SN to SgNg^{-1}} the 
set of $S^1$-fixed points
$\Hess(\sigma N \sigma^{-1},h)^{S^1}$ is precisely the
$\sigma$-translate \[\sigma \cdot \Hess(N,h)^{S^1} \subseteq S_n\] of
the $S^1$-fixed points of $\Hess(N,h)$.

In Section~\ref{sec:pinball} we focus attention on a
choice of Springer variety $\mathcal{S}_{\sigma N \sigma^{-1}}$
specified by $\lambda = (n-2,2)$ with nilpotent matrix $N$ and the
choice of permutation $\sigma$ 
determined by the rotated English filling.
In this setting 
we give below an explicit computation of the conjugate circle subgroup
$\sigma S^1 \sigma^{-1}$ and also the associated linear projection
$\Lie(T)^* \to \Lie(S^1)^*$. We illustrate with a concrete example.

\begin{example}
  Let $\lambda = (4,2)$. Then the corresponding matrix
  in standard Jordan canonical form is 
\[
N = \begin{bmatrix} 0 & 1 & 0 & 0 & 0 & 0 \\
                                0 & 0 & 1 & 0 & 0 & 0 \\
                                0 & 0 & 0 & 1 & 0 & 0 \\
                                0 & 0 & 0 & 0 & 0 & 0 \\
                                0 & 0 & 0 & 0  & 0 & 1 \\
                                0 & 0 & 0 & 0 & 0 & 0 
     \end{bmatrix} 
\]
and the associated 
permutation determined from the rotated English filling is $\sigma = 245613$. 
The standard $S^1$ in~\eqref{eq:def S^1 for Springer} is then 
conjugated to the circle subgroup 
\[
S^1 \cong \sigma S^1 \sigma^{-1} = \left\{ 
  \begin{bmatrix}
    t^2 & 0 & 0 & 0 & 0 & 0 \\
                                0 & t^6 & 0 & 0 & 0 & 0 \\
                                0 & 0 & t & 0 & 0 & 0 \\
                                0 & 0 & 0 & t^5 & 0 & 0 \\
                                0 & 0 & 0 & 0  & t^4 & 0 \\
                                0 & 0 & 0 & 0 & 0 & t^3 
     \end{bmatrix} 
\right\}.
\]
The corresponding 
linear projection $\Lie(T^6)^* = \t^* \to \Lie(S^1)^*$ induced
by the inclusion $S^1 \cong \sigma S^1 \sigma^{-1} \into T^6$ is given
by 
\begin{equation}\label{eq:new projection to S1} 
t_1 \mapsto  2t, \quad 
t_2 \mapsto  6t, \quad
t_3 \mapsto t, \quad
t_4 \mapsto 5t, \quad 
t_5 \mapsto  4t, \quad
t_6 \mapsto  3t.
\end{equation}
where $t$ denotes the variable in $\Lie(S^1)$ and the $t_i$ the
variables in $\Lie(T^6) \cong \R^6$. 
\end{example}

The general computation follows. 

\begin{lemma}\label{lemma:conjugated circle}
Let $n \geq 4$. Let $\lambda = (n-2,2)$ and let
$S^1$ denote the standard circle subgroup in~\eqref{eq:def S^1 for
  Springer}. 
Then the permutation $\sigma$ determined by the rotated English filling of
$\lambda$ is 
\[
\sigma = 2\, 4\, 5\, 6\, 7 \cdots n-1 \, n \hsm 1 \, 3
\]
in one-line notation and 
the conjugated subgroup $\sigma S^1 \sigma^{-1}$ is given by 
\begin{equation}\label{eq:conjugated circle} 
S^1 \cong \sigma S^1 \sigma^{-1} = \left\{
\begin{bmatrix}
t^2 & 0 & 0 & 0 & 0 & \cdots & 0 \\
 & t^{n} &  & & & & \\
 &   & t & & & & \\
 &   &   & t^{n-1} & & & \\
 &   &   &           & t^{n-2} & & \\
  &   &  &           &      & \ddots & \\
 & & & & & & t^3 
\end{bmatrix} \right\} \subseteq T.
\end{equation}
Moreover, the linear projection $\t^* \to \Lie(S^1)^*$ determined by the
inclusion of this circle subgroup $S^1 \cong \sigma S^1 \sigma^{-1} \into T$
is given by 
\begin{equation}\label{eq: projection to new torus (n,2) case}
t_1 \mapsto  2t, \quad 
t_2 \mapsto  nt, \quad 
t_3 \mapsto  t, \quad \textup{ and } 
t_k \mapsto (n+3-k) t,  \textup{ for } 4 \leq k \leq n.
\end{equation}
\end{lemma}

\begin{proof}[Proof of Lemma~\ref{lemma:conjugated circle}]
By definition $\lambda$ is the partition with $n-2$ boxes in the first
row and $2$ boxes in the second row. 
Its rotated English filling is 
$\myvcenter
{\def\lr#1{\multicolumn{1}{|@{\hspace{.6ex}}c@{\hspace{.6ex}}|}{\raisebox{-.3ex}{$#1$}}}
\raisebox{-.6ex}{$\begin{array}[b]{ccccccc}
\cline{1-1}\cline{2-2}\cline{3-3}\cline{4-4}\cline{5-5}\cline{6-6}\cline{7-7}
\lr{2}&\lr{4}&\lr{5}&\lr{\cdots}&\lr{n-2}&\lr{n-1}&\lr{n}\\
\cline{1-1}\cline{2-2}\cline{3-3}\cline{4-4}\cline{5-5}\cline{6-6}\cline{7-7}
\lr{1}&\lr{3}\\
\cline{1-1}\cline{2-2}
\end{array}$}
}
$
from which the form of $\sigma$ (obtained by the English reading of the
above filling) follows. 
Moreover the inverse of the given $\sigma$ is 
$\sigma^{-1} = n-1 \hsm 1 \hsm n \hsm 2 \hsm 3 \hsm 4
\cdots n-2$. The result follows by computation.
\end{proof}

\section{$S^1$-fixed points in Hessenberg varieties and permissible
  fillings}\label{sec:S1 action}

In this section we give an explicit bijection from the $S^1$-fixed points
of $\Hess(N,h)$, for various choices of $N$, to the set of
\textbf{permissible fillings} of $\lambda_N$. The last result of the
section, Corollary~\ref{corollary:fixed points and rotated english},
is used in Sections~\ref{sec:betti}-\ref{sec:pinball} but the
discussion is also of independent interest. Our results further
develop some ideas in \cite{Tym06}, in which Tymoczko constructs a
paving-by-affines of a nilpotent Hessenberg variety $\Hess(N,h)$ by
using certain Schubert cells. (In \cite{Tym06} Tymoczko considers more general
Hessenberg varieties but we focus on the nilpotent case here.) Since
each Schubert cell $BwB$ in $GL(n,\C)$ contains a unique coset $wB$
with $w$ a permutation matrix, it follows from her construction that
there is a unique such $w$ associated to each of the affine cells in
her paving of $\Hess(N,h)$, which in turn can be encoded in a filling
of a Young diagram \cite[Theorem 7.1]{Tym06}. 
Our main result in this
section, 
Theorem~\ref{theorem: fixed points and permissible fillings}, 
is another interpretation of this bijection; our main contribution is
to make
more explicit and precise the bijective correspondence between
the 
permissible fillings of $\lambda_N$ and the cosets $wB$ for $w$ 
a permutation matrix which lie in $\Hess(\sigma N \sigma^{-1}, h)$
(thought of as $S^1$-fixed points of $\Hess(\sigma N \sigma^{-1},h)$)
for
different choices of conjugates $\sigma N \sigma^{-1}$. 
We also refer the reader to \cite{BayHar10b}
for related discussion; in particular, Corollary~\ref{corollary:fixed
  points and rotated english} proves a claim used in \cite[Section
2]{BayHar10b}.

We begin by defining
permissible fillings following
\cite{Mbirika:2010}.

\begin{definition}\label{definition:permissible filling}
  Let $\lambda$ be a Young diagram with $n$ boxes and $h:
  \{1,2,\ldots, n\} \to \{1,2,\ldots,n\}$ a Hessenberg function.  A
  filling of $\lambda$ is a \textbf{$(h,\lambda)$-permissible filling}
  if for every
  horizontal adjacency \(\begin{array}{|c|c|} \cline{1-2} k & j \\
    \cline{1-2} \end{array}\) we have \(k \leq h(j).\) (When the $h$
  and $\lambda$ are understood from context we sometimes omit the
  $(h,\lambda)$ from terminology and refer simply to
  \textbf{permissible fillings}.) 
\end{definition}

\begin{remark} 
In the context of Springer varieties, for which $h(j)=j$ for all $j$,
the condition $k \leq h(j)$ becomes $k \leq j$. Thus in this case permissible fillings are
precisely the \textbf{row-strict} fillings. 
\end{remark}

Given $\lambda$ and $h$, we denote by 
\[
\PFill(\lambda,h) \subseteq \mathcal{F}i \ell \ell(\lambda) 
\]
the set of permissible fillings of $\lambda$. 
Let $N$ be a nilpotent $n \times n$ matrix in Jordan canonical form
with corresponding Young diagram $\lambda$, 
and let $h: \{1,2,\ldots,n\} \to \{1,2,\ldots, n\}$ be a Hessenberg function. 
Our goal is to construct an explicit identification between Hessenberg
fixed points $\Hess(\sigma N \sigma^{-1}, h)$ and permissible fillings
$\PFill(\lambda,h)$ for any permutation matrix $\sigma$. 

As a first step we define an identification between $S_n$ and
$\Fill(\lambda)$ which depends on the choice of permutation
$\sigma$. Recall that $\phi_\lambda: \Fill(\lambda) \to S_n$ is the
mapping given by the English reading of a filling. 

\begin{definition}\label{definition:phi lambda sigma} 
  Let $\sigma$ be a permutation in $S_n$ and $\lambda$ a Young diagram
  with $n$ boxes. Consider the filling $\phi_\lambda^{-1}(\sigma)$ of
  $\lambda$ corresponding to $\sigma$ via the English reading. The
  filling $\phi_\lambda^{-1}(\sigma)$ specifies a linear ordering on
  the boxes of $\lambda$. Define the map 
\begin{equation}\label{eq:def phi lambda sigma} 
\phi_{\lambda,\sigma}: \Fill(\lambda) \to S_n
\end{equation}
by associating to any filling $T$ of $\lambda$ the permutation whose
one-line notation is the reading of the entries of $T$ with respect to the linear
ordering given by $\phi_\lambda^{-1}(\sigma)$. 
\end{definition}

\begin{example}
  Suppose $\lambda = (3,2,1)$ and $\sigma = 253416$. Then
  $\phi_\lambda^{-1}(\sigma)$ is the filling 
$
\myvcenter
{\def\lr#1{\multicolumn{1}{|@{\hspace{.6ex}}c@{\hspace{.6ex}}|}{\raisebox{-.3ex}{$#1$}}}
\raisebox{-.6ex}{$\begin{array}[b]{ccc}
\cline{1-1}\cline{2-2}\cline{3-3}
\lr{2}&\lr{5}&\lr{3}\\
\cline{1-1}\cline{2-2}\cline{3-3}
\lr{4}&\lr{1}\\
\cline{1-1}\cline{2-2}
\lr{6}\\
\cline{1-1}
\end{array}$}}
$
so for the filling 
$
\myvcenter
T = {\def\lr#1{\multicolumn{1}{|@{\hspace{.6ex}}c@{\hspace{.6ex}}|}{\raisebox{-.3ex}{$#1$}}}
\raisebox{-.6ex}{$\begin{array}[b]{ccc}
\cline{1-1}\cline{2-2}\cline{3-3}
\lr{4}&\lr{1}&\lr{6}\\
\cline{1-1}\cline{2-2}\cline{3-3}
\lr{2}&\lr{3}\\
\cline{1-1}\cline{2-2}
\lr{5}\\
\cline{1-1}
\end{array}$}
}
$
the reading $\phi_{\lambda, \sigma}(T)$ would yield $346215$. 
\end{example}

\begin{remark}\label{remark:english and rotated english}
  By definition the mapping $\phi_{\lambda, id}$ corresponding to
  $\sigma=id$ the identity permutation coincides with the map
  $\phi_\lambda$ obtained via the English reading. Similarly the
  permutation $\sigma$ for which $\phi_{\lambda,\sigma}(T)$ is the
  rotated English reading is precisely the permutation corresponding
  under $\phi_\lambda$ to the rotated English filling of $\lambda$. 
\end{remark}

Remark~\ref{remark:english and rotated english} shows that both the
English and the rotated English readings of $\Fill(\lambda)$ are
special cases of $\phi_{\lambda,\sigma}$. The point of
Definition~\ref{definition:phi lambda sigma} is to emphasize that
other choices, corresponding to different choices of translated Hessenberg
varieties, are possible. 
We need the following lemma. 

\begin{lemma}\label{lemma:phi lambda sigma} 
  Let $\lambda$ be a Young diagram with $n$ boxes and $\sigma, \tau
  \in S_n$. Then 
\[
\phi^{-1}_{\lambda, \sigma}(\tau) = \phi^{-1}_\lambda(\tau \sigma).
\]
\end{lemma}

\begin{proof}
This follows from the definition of $\phi_{\lambda,\sigma}$ and the
fact that multiplication by $\sigma$ on the right re-orders the
entries in the one-line notation for $\tau$ precisely by replacing the
$i$-th entry $\tau(i)$ by $\tau(\sigma(i))$ for all $i$. 
\end{proof}

The main theorem of this section is the following. We
consider $\Hess(\sigma N \sigma^{-1}, h)^{S^1}$ to be a subset of
$S_n$ and $\PFill(\lambda)$ to be a subset of $\Fill(\lambda)$. 

\begin{theorem}\label{theorem: fixed points and permissible fillings}
  Let $N$ be an $n \times n$ nilpotent matrix in Jordan canonical form
  with weakly decreasing sizes of Jordan blocks 
  with respect to the standard basis of $\C^n$ and let $h:
  \{1,2,\ldots,n\} \to \{1,2,\ldots,n\}$ be a Hessenberg function. Let
  $\PFill(\lambda,h)$ denote the corresponding set of permissible
  fillings of $\lambda$. Let $\sigma \in S_n$ and denote by
  $\Hess(\sigma N \sigma^{-1}, h)$ the associated the nilpotent
  Hessenberg variety equipped with the $S^1$-action described in
  Section~\ref{sec:springer and circle action}. Then the assocation 
  \begin{equation}
    \label{eq:def Phi lambda sigma}
    \Phi_{\lambda, \sigma}: w \mapsto \phi^{-1}_{\lambda, \sigma}(w^{-1})
  \end{equation}
 defines a bijection from $\Hess(\sigma N \sigma^{-1}, h)^{S^1}$ to
 $\PFill(\lambda, h)$. 
\end{theorem}

In the proof of Theorem~\ref{theorem: fixed points and permissible
  fillings} we use the following terminology. Suppose $h:
\{1,2,\ldots,n\} \to \{1,2,\ldots, n\}$ is a Hessenberg function. We
define the \textbf{Hessenberg space} $H$ corresponding to $h$ to be the subspace of
$\mathfrak{gl}(n,\C)$ defined by 
\begin{equation}
  \label{eq:Hessenberg space}
  H := \{ X \in \mathfrak{gl}(n,\C) \, \vert \, X_{ij} = 0 \textup{ if
  } i > h(j) \}
\end{equation}
where $X_{ij}$ denotes the $(i,j)$-th entry of the matrix $X$. 

\begin{example}
  Suppose $h = (2,3,4,4)$. Then 
\[
H = \{ X \in \mathfrak{gl}(4,\C) \, \vert \, X_{3,1} = X_{4,1} =
X_{4,2} = 0 \} = \left\{ \begin{bmatrix} \star & \star & \star & \star
    \\ \star & \star & \star& \star \\ 0 & \star & \star & \star \\ 0
    & 0 & \star & \star \end{bmatrix} \right\} \subseteq
\mathfrak{gl}(n,\C)
\]
where the $\star$ denotes free variables. 
\end{example}

It is straightforward to reformulate the
definition~\eqref{eq:def-Hess} of Hessenberg varieties as follows: for
a given Hessenberg function $h$ with corresponding Hessenberg space
$H$, 
\begin{equation}
  \label{eq:def-Hess-v2}
  \Hess(N,h) = \{[g] \in GL(n,\C)/B \, \vert \, g^{-1} N g \in H \}.
\end{equation}
In particular, the $S^1$-fixed points of $\Hess(N,h)$ are precisely 
\begin{equation}
  \label{eq:fixed points Hess}
  \Hess(N,h)^{S^1} \cong \{w \in S_n \, \vert \, w^{-1} N w \in H \}.
\end{equation}
We use the following lemma. 

\begin{lemma}\label{lemma:permissible versus hessenberg space} 
  Let $\lambda$ be a Young diagram with $n$ boxes and $h:
  \{1,2,\ldots,n\} \to \{1,2,\ldots,n\}$ a Hessenberg function with
  corresponding Hessenberg space $H$. Let $T$ be a filling of
  $\lambda$ by the alphabet $[n]$ and let $M$ be the $n \times n$
  matrix obtained by applying the adjacency algorithm to $T$. Then 
\[
T \textup{ is } (h,\lambda)\textup{-permissible} \quad
\Longleftrightarrow \quad M \in H.
\]
\end{lemma}

\begin{proof}
  By definition of the adjacency algorithm, the $(i,j)$-th entry of
  $M$ is non-zero precisely when \(\begin{array}{|c|c|} \cline{1-2} i & j \\
    \cline{1-2} \end{array}\) occurs in the filling of $T$. Hence by
  definition of $H$ the matrix $M$ is in $H$ precisely if, for all
  such adjacent pairs $(i,j)$ in $T$, we have $i \leq h(j)$. This is
  exactly the definition of a $(h,\lambda)$-permissible filling. 
\end{proof}

\begin{proof}[Proof of Theorem~\ref{theorem: fixed points and permissible fillings}]
  We first prove the claim for the special case $\sigma = id$. In this
  case $\phi_{\lambda,id} = \phi_\lambda$
  (cf. Remark~\ref{remark:english and rotated english}) and we wish to
  show that the association $w \mapsto \Phi_{\lambda, id}(w) :=
  \phi_{\lambda}^{-1}(w^{-1})$ defines a bijection between
  $\Hess(N,h)^{S^1}$ and $\PFill(\lambda,h)$. Since taking inverses is
  a bijection on $S_n$ and $\phi_\lambda$ is also a bijection from
  $\Fill(\lambda)$ to $S_n$, the content of the claim is that a
  permutation $w$ is in $\Hess(N,h)^{S^1}$ precisely when the filling
  $\phi_\lambda^{-1}(w^{-1})$ is permissible. Recall
  from~\eqref{eq:fixed points Hess} that 
\[
w \in \Hess(N,h)^{S^1} \quad \Longleftrightarrow \quad w^{-1}Nw \in H.
\]
By Theorem~\ref{adjacency}, the matrix $w^{-1}Nw$ is
precisely the adjacent-pair matrix for 
the filling $\phi_\lambda^{-1}(w^{-1})$. The claim now follows from
Lemma~\ref{lemma:permissible versus hessenberg space}.

The claim for $\Phi_{\lambda,\sigma}$ for arbitrary $\sigma \in S_n$
follows from the special case $\Phi_{\lambda,id}$ because 
\[
\Hess(\sigma N\sigma^{-1},h)^{S^1} = \sigma \cdot \Hess(N,h)^{S^1}
\]
and 
\[
\phi_{\lambda,\sigma}^{-1}((\sigma \cdot w)^{-1}) =
\phi_{\lambda,\sigma}^{-1}(w^{-1}\sigma^{-1}) =
\phi_{\lambda}^{-1}(w^{-1})
\]
where the last equality uses Lemma~\ref{lemma:phi lambda sigma}. This completes the
proof. 
\end{proof}

The following is used below in
Sections~\ref{sec:betti}-\ref{sec:pinball} as well as in
\cite{BayHar10b}. Given a Young diagram $\lambda$ with $n$ boxes,
denote by $T_{RE}$ the rotated English filling of $\lambda$. 

\begin{corollary}\label{corollary:fixed points and rotated english}
   Let $N$ be an $n \times n$ nilpotent matrix in Jordan canonical
   form and weakly decreasing sizes of Jordan blocks 
  with respect to the standard basis of $\C^n$ and let $h:
  \{1,2,\ldots,n\} \to \{1,2,\ldots,n\}$ be a Hessenberg function. Let
  $\PFill(\lambda,h)$ denote the corresponding set of permissible
  fillings of $\lambda$. Let $\sigma = \phi_\lambda(T_{RE})$ be the
  permutation corresponding to the rotated English filling of
  $\lambda$. 
  Then 
  \begin{equation}
    \label{eq:def Phi lambda sigma}
    \Phi_{\lambda, \sigma}: w \mapsto \phi^{-1}_{\lambda, \sigma}(w^{-1})
  \end{equation}
 is a bijection from $\Hess(\sigma N \sigma^{-1}, h)^{S^1}$ to
 $\PFill(\lambda, h)$. 
\end{corollary}

\section{Betti-acceptable pinball and linear independence}\label{sec:betti}

For the rest of the manuscript we restrict attention to nilpotent
Springer varieties, i.e., the case in which the Hessenberg function is
the identity function $h(i)=i$ for all $1 \leq i \leq n$.  In this section we
recount for the convenience of the reader several ideas developed in
\cite{HarTym10, BayHar10b} which are used in the next two
sections in the study of a special class of nilpotent Springer varieties. 
First we recall the \textbf{dimension pair
  algorithm} introduced in \cite{BayHar10b} 
which associates to each $S^1$-fixed point in a
nilpotent Hessenberg variety a permutation in $S_n$. We also recall
the interpretation of the algorithm in terms of the poset pinball game
introduced in \cite{HarTym10}.  More specifically, in the case of
nilpotent Springer varieties, the algorithm has an interpretation as
producing the output of a \textbf{successful game of Betti poset
  pinball}, as is shown in \cite[Proposition 3.6]{BayHar10b}.  We keep
exposition brief and refer the reader to \cite{BayHar10b} for details.

We begin with the definition of dimension pairs for the special case
of the identity Hessenberg function $h(i)=i$. 

\begin{definition}
Let $\lambda$ a Young diagram with
$n$ boxes
and $T$ a permissible filling of $\lambda$. 
The pair $(a, b)$ is a \textbf{dimension pair} of $T$ 
if the following conditions hold:
\begin{enumerate} 
\item \(b > a, \)
\item $b$ is either 
\begin{itemize}
\item below $a$ and in the same column, or 
\item anywhere in a column strictly to the left of the column of $a$, 
\end{itemize} 
and 
\item if there exists a box with filling $c$ directly adjacent to the right
of $a$, then $b \leq c$. 
\end{enumerate}
For a dimension pair $(a,b)$ of $T$, we will refer to $b$ as the
\textbf{top part} of the dimension pair. 
\end{definition}

\begin{example}
Let $\lambda = (2,2)$. 
For the permissible filling 
\[
{\def\lr#1{\multicolumn{1}{|@{\hspace{.6ex}}c@{\hspace{.6ex}}|}{\raisebox{-.3ex}{$#1$}}}
\raisebox{-.6ex}{$\begin{array}[b]{cc}
\cline{1-1}\cline{2-2}
\lr{1}&\lr{3}\\
\cline{1-1}\cline{2-2}
\lr{2}&\lr{4}\\
\cline{1-1}\cline{2-2}
\end{array}$}
}
\]
the dimension pairs are $\{(1,2),(3,4)\}$.
\end{example}

Given a permissible filling $T$ of $\lambda$, 
denote by $DP^T$ the set of dimension pairs of $T$. 
For each integer $\ell$ with $2 \leq \ell \leq n$, define 
\begin{equation}\label{eq:def xell}
x_\ell := \lvert \{ (a,\ell) \hsm \vert \hsm (a,\ell) \in DP^T \} \rvert.
\end{equation}
We call the integral vector $\mathbf{x} = (x_2, x_3, \ldots, x_n)$ the
\textbf{list of top parts} of $T$.  To each such $\mathbf{x}$ we
associate a permutation in $S_n$ as follows. As a preliminary step,
for each $\ell$ with $2 \leq \ell \leq n$ define
\[
u_{\ell}(\mathbf{x}) := \begin{cases} s_{\ell-1} s_{\ell-2} \cdots s_{\ell-x_\ell}
  \quad \textup{ if } x_\ell > 0 \\
1 \quad \textup{ if } x_\ell = 0 
\end{cases} 
\]
where $s_i$ denotes the simple transposition $(i,i+1)$ in $S_n$ and
$1$ denotes the identity permutation. We define an association
$\mathbf{x} \mapsto \omega(\mathbf{x}) \in S_n$ by
\begin{equation}\label{eq:permutation from top parts} 
\omega(\mathbf{x}) := u_2(\mathbf{x}) u_3(\mathbf{x})
\cdots u_n(\mathbf{x}) \in S_n.
\end{equation}

With the terminology in place we now recall the \textbf{dimension pair
algorithm} introduced in \cite{BayHar10b}. Suppose $N$ is a nilpotent $n \times n$ matrix in Jordan
canonical form and weakly decreasing sizes of Jordan blocks, with corresponding Young diagram $\lambda$. Following
notation in Section~\ref{sec:S1 action} denote by $T_{RE}$ the rotated
English filling of $\lambda$ and let $\sigma :=
\phi_\lambda(T_{RE})$ be the permutation such that $N_{hf} :=
\sigma N \sigma^{-1}$ is the choice of highest form of $N$ used in
\cite[Section 4]{Tym06}. 

\medskip
\noindent \textbf{Definition of $\roll: \mathcal{S}_{N_{hf}}^{S^1} \to
  S_n$:}

\begin{enumerate}
\item Let $w \in \Hess(N_{hf}, h)^{S^1}$ and let $\phi_{\lambda,\sigma}^{-1}(w^{-1})$ be its corresponding
  permissible filling. 
\item Let $DP^{\phi_{\lambda, \sigma}^{-1}(w^{-1})}$ be the set of dimension pairs in the permissible
  filling $\phi_{\lambda, \sigma}^{-1}(w^{-1})$. 
\item For each $\ell$ with $2 \leq \ell \leq n$, 
set 
\[
x_\ell := \lvert \{  (a,\ell) \hsm \vert \hsm (a, \ell) \in
DP^{\phi_{\lambda, \sigma}^{-1}(w^{-1})} \} \rvert
\]
as in~\eqref{eq:def xell} and define \(\mathbf{x} := (x_2, \ldots, x_n).\) 
\item Define $\roll(w) := (\omega(\mathbf{x}))^{-1}$ where $\omega(\mathbf{x})$
  is the permutation associated to the integer vector $\mathbf{x}$
  defined in~\eqref{eq:permutation from top parts}.
\end{enumerate}

We call $\roll(w)$ the \textbf{rolldown} of $w$, following terminology
introduced in \cite{HarTym10}. The idea motivating the dimension pair
algorithm is that we can interpret the association $w \mapsto
\roll(w)$ as a result of a game of \textbf{poset pinball}, defined in
\cite[Section 3]{HarTym10}. A poset pinball game starts with the data
of an ambient partially ordered set, a rank function $\ell$ on the poset, and
a designated subset  of the poset (called
the \textbf{initial subset}); in our setting these are the permutation
group $S_n$ equipped with Bruhat order, rank function $\ell: S_n \to
\Z$ given by Bruhat length, and the Springer fixed points
$\mathcal{S}_{N_{hf}}^{S^1}$ respectively. The \textbf{Betti pinball}
version of the game then proceeds by
assigning to each element of $\mathcal{S}_{N_{hf}}^{S^1}$ a
permutation in $S_n$ satisfying certain conditions (see \cite[Section
3]{HarTym10} for details), one of which concerns the Betti numbers of
$\mathcal{S}_{N_{hf}}$. We recall the following 
result
of Tymoczko (reformulated in our language). Although in \cite{Tym06}
Tymoczko deals with a more general situation we state her result only for the 
special case of Springer varieties.  The statement assumes
that $N_{hf}$ is in the highest form 
corresponding to the rotated English filling of
Definition~\ref{our filling}. 

\begin{theorem}\label{theorem:paving} \textbf{(\cite[Theorem
    1.1]{Tym06})}
Let 
$N_{hf}: \C^n \to \C^n$ be a nilpotent matrix 
in highest form chosen as above
and let $\lambda := \lambda_{N_{hf}}$. Let 
$\mathcal{S}_{N_{hf}}$ denote the corresponding nilpotent Springer
variety. 
There is a paving by (complex) affine cells of $\mathcal{S}_{N_{hf}}$ 
such that: 
  \begin{itemize}
  \item the affine cells are in one-to-one correspondence with
    $\mathcal{S}_{N_{hf}}^{S^1}$, and
 \item the (complex) dimension of the affine cell $C_w$ corresponding to a
    fixed point \(w \in \mathcal{S}_{N_{hf}}^{S^1}\) is 
\begin{equation}\label{eq:dim C_w}
\dim_{\C}(C_w) =  \lvert DP^{\phi_{\lambda, \sigma}^{-1}(w^{-1})} \rvert
\end{equation}
where $\sigma = \phi_\lambda(T_{RE})$. 
\end{itemize}
\end{theorem}

In particular, Theorem~\ref{theorem:paving} implies that the odd Betti
numbers of $\mathcal{S}_{N_{hf}}$ are $0$, and the $2k$-th even Betti
number is precisely the number of fixed points $w$ such that $\lvert
DP^{\phi_{\lambda}^{-1}(w^{-1})} \rvert = k$.  In this sense the
dimension pairs in the permissible fillings
$\phi_{\lambda,\sigma}^{-1}(w^{-1})$ contain the data of the Betti
numbers of $\mathcal{S}_{N_{hf}}$.
One of the rules of the Betti pinball game (see \cite[Section
3]{HarTym10} and \cite[Section 3]{BayHar10b} for more details) is that 
for every $k \geq 0, k \in \Z$, we must have 
\[
b_k = \left\lvert \left\{\roll(w) \hsm \vert \hsm  w \in \Hess(N_{hf},
    h)^{S^1} \textup{ with }
\ell(\roll(w)) = k \right\} \right\rvert
\]
where $b_k$ denotes the $2k$-th Betti number of
$\mathcal{S}_{N_{hf}}$. 
By the definition of $\roll: \mathcal{S}_{N_{hf}}^{S^1} \to S_n$ this
condition is satisfied. 
It is shown in \cite[Proposition 3.8]{BayHar10b} that the 
association $w \mapsto \roll(w)$ also satisfies the other necessary
conditions to be 
interpreted in this context as an outcome of a successful game of Betti poset
  pinball.

For any $u \in S_n$, define the class $p_u := \pi(\sigma_u)$ to be the image
of the classical equivariant Schubert class $\sigma_u \in
H^*_T(\Flags(\C^n))$ 
under the projection 
\begin{equation}\label{eq:projection}
\pi: H^*_{T}(\Flags(\C^n)) \to H^*_{S^1}(\mathcal{S}_{N_{hf}})
\end{equation}
induced by the inclusion of groups $S^1 \into T$ and the 
$S^1$-equivariant inclusion of
spaces $\mathcal{S}_{N_{hf}} \into \Flags(\C^n)$. In analogy with the
terminology in \cite{HarTym09, BayHar10b}, we refer to the images
$p_u$ as \textbf{Springer Schubert classes}.

One of the goals of poset pinball is to build explicit
module bases for equivariant cohomology rings. In the context of
nilpotent Hessenberg varieties, one method by which to
do so is to find an appropriate subset of the Hessenberg Schubert
classes $p_u$ (defined analogously to the Springer Schubert classes
above) which form a module basis for
$H^*_{S^1}(\Hess(N_{hf},h))$. To show that a subset is a basis,
we must in particular show that the subset is linearly independent. 
 Using the fact that equivariant Schubert classes satisfy 
\begin{equation}
  \label{eq:schuberts upper triangular}
  \sigma_v(w) = 0 \textup{ if } w \not \geq v
\end{equation}
for all $w,v \in S_n$, it follows that if the rolldowns $\roll(w)$ of
the Hessenberg fixed points satisfy the 
\textbf{poset-upper-triangularity condition} 
\begin{equation}\label{eq:inequality}
\roll(w) \leq u \Leftrightarrow w \leq u
\end{equation}
for all $w, u \in \mathcal{S}_{N_{hf}}^{S^1}$, then the corresponding
Hessenberg Schubert classes are linearly independent \cite[Section
2]{HarTym10}. 
In \cite{BayHar10b} the
results of the dimension pair algorithm is studied in detail for a 
special case of regular nilpotent Hessenberg varieties
$\Hess(N,h)$. In this case it turns out that the set of permutations
$\{\roll(w)\}_{w \in \Hess(N,h)^{S^1}}$ satisfy
the poset-upper-triangularity property~\eqref{eq:inequality} (see \cite[Theorem 4.1]{BayHar10b}). 
Combining this poset-upper-triangularity with the fact that the
rolldowns obtained by the dimension pair algorithm are compatible with the Betti numbers of
$\Hess(N_{hf},h)$ \cite[Lemma 3.6]{BayHar10b}, 
it then follows from \cite[Proposition
4.14]{HarTym10} that the corresponding Hessenberg Schubert classes
$\{p_{\roll(w)}\}_{w \in \Hess(N,h)^{S^1}}$ form a $H^*_{S^1}(\pt)$-module basis for
$H^*_{S^1}(\Hess(N,h))$ \cite[Proposition 3.9]{BayHar10b}. 

However, it turns out that in
the $(n-2,2)$ Springer variety case studied in detail below, the rolldowns
$\{\roll(w)\}_{w \in \mathcal{S}_{N_{hf}}^{S^1}}$ coming from the
dimension pair algorithm are not
necessarily poset-upper-triangular, as we show below, so we cannot
apply \cite[Proposition 4.14]{HarTym10}. Instead it requires further
analysis to determine that the classes $\{p_{\roll(w)}\}_{w \in
  \mathcal{S}_{N_{hf}}^{S^1}}$ are linearly independent; this is the
content of Section~\ref{sec:pinball} below. Once linear independence
is established we use the following proposition to conclude that the
set $\{p_{\roll(w)}\}_{w \in \mathcal{S}_{N_{hf}}^{S^1}}$ is a module
basis.

\begin{proposition}\label{proposition: if triangular then basis}
Let $N:\C^n \to \C^n$ be a nilpotent operator in standard Jordan
canonical form with weakly decreasing Jordan block sizes with
corresponding Young diagram $\lambda$. Let
$\mathcal{S}_{N_{hf}}$ be the Springer
variety corresponding to the highest form $N_{hf} := \sigma N \sigma^{-1}$ where
$\sigma$ is the permutation corresponding to the rotated English
filling of $\lambda$, 
equipped with the $S^1$ action defined in~\eqref{eq:conjugated circle}. 
Let $\roll: \mathcal{S}_{N_{hf}}^{S^1} \to S_n$ be the dimension-pair
algorithm defined above. Suppose the classes 
$\{p_{\roll(w)} \hsm
\vert \hsm w \in \mathcal{S}_{N_{hf}}^{S^1}\}$ are linearly independent in
$H^*_{S^1}(\mathcal{S}_{N_{hf}})$. Then the set $\{p_{\roll(w)} \hsm
\vert \hsm w \in \mathcal{S}_{N_{hf}}^{S^1}\}$ of Springer Schubert classes form a $H^*_{S^1}(\pt)$-module
basis for the $S^1$-equivariant cohomology ring
$H^*_{S^1}(\mathcal{S}_{N_{hf}})$. 
\end{proposition}

\begin{proof} 
Since $\roll: \mathcal{S}_{N_{hf}}^{S^1} \to S_n$ represents a
possible outcome of a successful game of Betti poset pinball by
\cite[Proposition 3.7]{BayHar10b} 
 the assertion follows from \cite[Proposition 4.13]{HarTym10}.
\end{proof}

\begin{remark}
  The $(n-2,2)$ Springer variety example studied here is the first
  example in the poset pinball literature of an instance of
  successful Betti pinball which does not yield a
  poset-upper-triangular basis. 
\end{remark}

Finally we briefly recall the injectivity results in
equivariant cohomology 
which computationally simplify the proof
that the Springer Schubert
classes are linearly independent. The next proposition follows from
known results about the topology of Springer varieties \cite{Spa76}
and a standard
argument in equivariant cohomology 
(see e.g. \cite[Remark 4.11 and Proposition 6.2]{HarTym10}). 

\begin{proposition}\label{prop:linear independence}
Let $N:\C^n \to \C^n$ be a nilpotent operator in standard Jordan
canonical form with weakly decreasing Jordan block sizes with
corresponding Young diagram $\lambda$. Let
$\mathcal{S}_{N_{hf}}$ be the Springer
variety corresponding to the highest form $N_{hf} := \sigma N \sigma^{-1}$ where
$\sigma$ is the permutation corresponding to the rotated English
filling of $\lambda$, 
equipped with the $S^1$-action defined in~\eqref{eq:conjugated circle}. 
Then the inclusion $\iota:
\mathcal{S}_{N_{hf}}^{S^1} \into \mathcal{S}_{N_{hf}}$ induces an
injection in $S^1$-equivariant
cohomology 
\[
\iota^* \colon  H^*_{S^1}(\mathcal{S}_{N_{hf}}) \into
H^*_{S^1}(\mathcal{S}_{N_{hf}}^{S^1}) \cong \bigoplus_{w \in
  \mathcal{S}_{N_{hf}}^{S^1}} H^*_{S^1}(\pt) \cong \bigoplus_{w \in
  \mathcal{S}_{N_{hf}}^{S^1}} \C[t].
\]
\end{proposition}

The above proposition implies that a Springer Schubert
class $p_u$ in $H^*_{S^1}(\mathcal{S}_{N_{hf}})$ can be
specified by $\iota^*(p_u)$, which we view as a vector of 
polynomials in $\C[t]$ with coordinates indexed by the fixed points
$\mathcal{S}_{N_{hf}}^{S^1}$. Following notation of \cite{HarTym09,
  HarTym10, BayHar10b}, we denote by $p_u(w)$ the $w$-th coordinate of
$\iota^*(p_u)$. 
The next result, which we use later,
is straightforward.

\begin{proposition}\label{prop: matrix form}
Let $N, \lambda, \mathcal{S}_{N_{hf}}$ be as above. If the columns of the matrix 
\[
(p_{\roll(w)}(u))_{w, u \in \mathcal{S}_{N_{hf}}^{S^1}}
\]
(where the variable $w$ is the index of the
columns and $u$ the index of the rows) 
are linearly independent over $H^*_{S^1}(\pt) \cong \C[t]$, then 
the set of Springer Schubert classes
$\{p_{\roll(w)}\}_{w \in \mathcal{S}_{N_{hf}}^{S^1}}$ is linearly
independent. 
\end{proposition}

\section{Small-$n$ cases: $n=4$ and $n=5$}\label{sec:examples}

In this section and Section~\ref{sec:pinball} we restrict attention to the nilpotent
Springer varieties corresponding to Young diagrams of the form
$(n-2,2)$ for $n\geq 4$. In this setting we denote by
$\mathcal{S}_{(n-2,2)}$ the Springer variety $\mathcal{S}_{N_{hf}}$ corresponding to the
nilpotent matrix $N_{hf} :=\sigma N \sigma^{-1}$ 
in highest form with associated
Young diagram $(n-2,2)$ where $\sigma$ is the permutation
corresponding to the rotated English filling of $(n-2,2)$. The goal, as explained in
Section~\ref{sec:betti}, is to prove that the dimension pair algorithm
produces in this case a module basis for
$H^*_{S^1}(\mathcal{S}_{(n-2,2)})$. 
To this end we concretely compute the Springer fixed points,
associated permissible fillings, dimension pairs, and rolldowns 
for the cases $n=4$ 
and $n=5$, i.e. for the Springer varieties corresponding to the Young
diagrams 
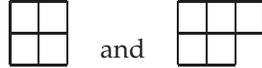
\[
{\def\lr#1{\multicolumn{1}{|@{\hspace{.6ex}}c@{\hspace{.6ex}}|}{\raisebox{-.3ex}{$#1$}}}
\raisebox{-.6ex}{$\begin{array}[b]{cc}
\cline{1-1}\cline{2-2}
\lr{\phantom{1}}&\lr{\phantom{1}}\\
\cline{1-1}\cline{2-2}
\lr{\phantom{1}}&\lr{\phantom{1}}\\
\cline{1-1}\cline{2-2}
\end{array}$}
}
 \quad \textup{ and } \quad 
{\def\lr#1{\multicolumn{1}{|@{\hspace{.6ex}}c@{\hspace{.6ex}}|}{\raisebox{-.3ex}{$#1$}}}
\raisebox{-.6ex}{$\begin{array}[b]{ccc}
\cline{1-1}\cline{2-2}\cline{3-3}
\lr{\phantom{1}}&\lr{\phantom{1}}&\lr{\phantom{1}}\\
\cline{1-1}\cline{2-2}\cline{3-3}
\lr{\phantom{1}}&\lr{\phantom{1}}\\
\cline{1-1}\cline{2-2}
\end{array}$}
}
\]
We also 
explicitly check in these cases that the corresponding Springer Schubert classes are 
poset-upper-triangular and hence linearly independent. 
The inductive argument we give in the next section requires the $n=4$
case as its base case. We choose to additionally explicitly compute and record
the $n=5$ case because 
it suggests the 
outline of the general inductive argument.

Below we present two tables of data. The columns correspond 
to the following:

\begin{itemize}
\item $w$: an $S^1$-fixed point in the Springer variety
  $\mathcal{S}_{(n-2,2)}$. 
\item $w^{-1}$: the inverse of $w$.
\item $perm \hsm filling$: the permissible filling
  $\phi_{\lambda,\sigma}^{-1}(w^{-1})$. 
\item $dim \hsm pair$: the dimension pairs of the permissible filling.
\item $deg$: the number of dimension pairs of the permissible filling
(equivalently, the cohomology degree of the associated Springer Schubert class). 
\item $\omega(\mathbf{x})$ the permutation associated to the list
  $\mathbf{x}$ of ``top
  parts'' of the dimension pairs.
\item $\roll(w)$: inverse of $\omega(\mathbf{x})$, and by definition of the
  dimension pair algorithm, the rolldown of $w$.
\end{itemize}

\begin{example}
Let $n=4$ and $\lambda_N = (2,2)$. 
The following table records the data outlined above.  
Part of these computations are also contained in \cite{Mbirika:2010}.

\begin{table}[h] \centering
\caption{Dimension pair data for the Springer variety $\mathcal{S}_{(2,2)}$.}
\begin{tabular}{ | c | c | c | c | c | c |c| }\hline
$w$ & $w^{-1}$ & perm filling & dim pair & deg &  $\omega(\mathbf{x})$ &
$\roll(w)$ \\ \hline \hline

1234 $= e$ & 1234 & 
{\def\lr#1{\multicolumn{1}{|@{\hspace{.6ex}}c@{\hspace{.6ex}}|}{\raisebox{-.3ex}{$#1$}}}
\raisebox{-.6ex}{$\begin{array}[b]{cc}
\cline{1-1}\cline{2-2}
\lr{2}&\lr{4}\\
\cline{1-1}\cline{2-2}
\lr{1}&\lr{3}\\
\cline{1-1}\cline{2-2}
\end{array}$}
}
& $\emptyset$ & 0 & 1234 & 1234 $= e$ \\ \hline

2134  $= s_1$ & 2134 & 
{\def\lr#1{\multicolumn{1}{|@{\hspace{.6ex}}c@{\hspace{.6ex}}|}{\raisebox{-.3ex}{$#1$}}}
\raisebox{-.6ex}{$\begin{array}[b]{cc}
\cline{1-1}\cline{2-2}
\lr{1}&\lr{4}\\
\cline{1-1}\cline{2-2}
\lr{2}&\lr{3}\\
\cline{1-1}\cline{2-2}
\end{array}$}
}
& $\{(1,2)\}$ & 1 & 2134 & 2134 $=s_1$ \\ \hline

1324 $=s_2$ & 1324 & 
{\def\lr#1{\multicolumn{1}{|@{\hspace{.6ex}}c@{\hspace{.6ex}}|}{\raisebox{-.3ex}{$#1$}}}
\raisebox{-.6ex}{$\begin{array}[b]{cc}
\cline{1-1}\cline{2-2}
\lr{3}&\lr{4}\\
\cline{1-1}\cline{2-2}
\lr{1}&\lr{2}\\
\cline{1-1}\cline{2-2}
\end{array}$}
}
& $\{(2,3)\}$ & 1 & 1324 & 1324 $=s_2$ \\ \hline

1243 $=s_3$ & 1243 & 
{\def\lr#1{\multicolumn{1}{|@{\hspace{.6ex}}c@{\hspace{.6ex}}|}{\raisebox{-.3ex}{$#1$}}}
\raisebox{-.6ex}{$\begin{array}[b]{cc}
\cline{1-1}\cline{2-2}
\lr{2}&\lr{3}\\
\cline{1-1}\cline{2-2}
\lr{1}&\lr{4}\\
\cline{1-1}\cline{2-2}
\end{array}$}
}
& $\{(3,4)\}$ & 1 & 1243 & 1243 $=s_3$\\ \hline

2143 $=s_1 s_3$ & 2143 & 
{\def\lr#1{\multicolumn{1}{|@{\hspace{.6ex}}c@{\hspace{.6ex}}|}{\raisebox{-.3ex}{$#1$}}}
\raisebox{-.6ex}{$\begin{array}[b]{cc}
\cline{1-1}\cline{2-2}
\lr{1}&\lr{3}\\
\cline{1-1}\cline{2-2}
\lr{2}&\lr{4}\\
\cline{1-1}\cline{2-2}
\end{array}$}
}
& $\{(1,2),(3,4)\}$ & 2 & 2143 & 2143 $=s_1 s_3$ \\ \hline 

2413 $=s_1 s_3 s_2$ & 3142 & 
{\def\lr#1{\multicolumn{1}{|@{\hspace{.6ex}}c@{\hspace{.6ex}}|}{\raisebox{-.3ex}{$#1$}}}
\raisebox{-.6ex}{$\begin{array}[b]{cc}
\cline{1-1}\cline{2-2}
\lr{1}&\lr{2}\\
\cline{1-1}\cline{2-2}
\lr{3}&\lr{4}\\
\cline{1-1}\cline{2-2}
\end{array}$}
}
& $\{(2,4),(2,3)\}$ & 2 & 1342 & 1423 $=s_3 s_2$ \\ \hline 
\end{tabular}
\end{table}

From this table it can be seen explicitly that the only Springer fixed
point $w$ in $\mathcal{S}_{(2,2)}$ with $\roll(w)
\neq w$ is $w = 2413$. 
Moreover it is straightforward to
check that the rolldown $1423$ of $w = 2413$ is not Bruhat-less
than any of the other Springer fixed points. These facts together
imply that these Springer fixed points and associated rolldowns
satisfy the poset-upper-triangularity property 
\begin{equation}\label{eq:inequality-2}
\roll(w) \leq u \Leftrightarrow w \leq u
\end{equation}
for all fixed points $w, u$. By an argument identical
to \cite[Lemma 4.4]{BayHar10b} which uses the
poset-upper-triangularity property~\eqref{eq:schuberts upper triangular} of the equivariant
Schubert classes $\{\sigma_w\}_{w \in S_n}$, this implies that the Springer Schubert classes
$\{p_{\roll(w)}\}_{w \in \mathcal{S}_{(2,2)}^{S^1}}$ are
poset-upper-triangular and hence linearly independent and a
$H^*_{S^1}(\pt)$-module basis for $H^*_{S^1}(\mathcal{S}_{(2,2)})$. 
\end{example}

We have just explicitly checked that in the case $n=4$, the dimension
pair algorithm interpreted in terms of Betti pinball produces a module basis of
$H^*_{S^1}(\mathcal{S}_{(2,2)})$. We now compute the $n=5$ case and
relate it to the $n=4$ case, thereby illustrating 
the outline of the general inductive argument. 

\begin{example}\label{example:n=5}
Let $n=5$ and $\lambda = (3,2)$. 
Suppose $T$ is a permissible filling of $(3,2)$ where the 
entry $5$ is in the top row. Since the rows in a permissible filling are increasing this means
that the $5$ occurs in the rightmost box of the top row of $T$. Deleting this
box yields a valid permissible filling of $(2,2)$ which therefore
occurs in the previous $n=4$ example. For permissible fillings $T$ of
this form the corresponding fixed point
$w$ and its rolldown are easily seen to be identical to those
obtained in the previous example (viewed as elements of $S_5$ instead
of $S_4$ via the usual embedding $S_4 \into S_5$). Hence the permissible fillings in the $n=5$
case which do not occur in the $n=4$ case are precisely those for which the
entry $5$ is in the bottom row. There are four such permissible
fillings as may be seen in the table below. 

\begin{table}[h] \centering
\caption{Dimension pair data for the Springer variety $\mathcal{S}_{(3,2)}$.}
\begin{tabular}{ | c | c | c | c | c | c |c| }\hline
$w$  & $w^{-1}$ & perm filling & dim pair & deg &  $\omega(\mathbf{x})$ &
$\roll(w)$ \\ \hline \hline

12345 $=e$ & 12345 &  
{\def\lr#1{\multicolumn{1}{|@{\hspace{.6ex}}c@{\hspace{.6ex}}|}{\raisebox{-.3ex}{$#1$}}}
\raisebox{-.6ex}{$\begin{array}[b]{ccc}
\cline{1-1}\cline{2-2}\cline{3-3}
\lr{2}&\lr{4}&\lr{5}\\
\cline{1-1}\cline{2-2}\cline{3-3}
\lr{1}&\lr{3}\\
\cline{1-1}\cline{2-2}
\end{array}$}
}

 & $\emptyset$ &  0 & 12345 & 12345 $=e$ \\ \hline 

21345 $=s_1$ & 21345 &  
{\def\lr#1{\multicolumn{1}{|@{\hspace{.6ex}}c@{\hspace{.6ex}}|}{\raisebox{-.3ex}{$#1$}}}
\raisebox{-.6ex}{$\begin{array}[b]{ccc}
\cline{1-1}\cline{2-2}\cline{3-3}
\lr{1}&\lr{4}&\lr{5}\\
\cline{1-1}\cline{2-2}\cline{3-3}
\lr{2}&\lr{3}\\
\cline{1-1}\cline{2-2}
\end{array}$}
}

& $\{(1, 2)\}$ &  1& 21345 & 21345 $=s_1$\\ \hline

13245 $=s_2$ & 13245 &  
{\def\lr#1{\multicolumn{1}{|@{\hspace{.6ex}}c@{\hspace{.6ex}}|}{\raisebox{-.3ex}{$#1$}}}
\raisebox{-.6ex}{$\begin{array}[b]{ccc}
\cline{1-1}\cline{2-2}\cline{3-3}
\lr{3}&\lr{4}&\lr{5}\\
\cline{1-1}\cline{2-2}\cline{3-3}
\lr{1}&\lr{2}\\
\cline{1-1}\cline{2-2}
\end{array}$}
}

&  $\{(2, 3)\}$ &  1 & 13245 & 13245 $=s_2$ \\ \hline

12435 $=s_3$& 12435 &  
{\def\lr#1{\multicolumn{1}{|@{\hspace{.6ex}}c@{\hspace{.6ex}}|}{\raisebox{-.3ex}{$#1$}}}
\raisebox{-.6ex}{$\begin{array}[b]{ccc}
\cline{1-1}\cline{2-2}\cline{3-3}
\lr{2}&\lr{3}&\lr{5}\\
\cline{1-1}\cline{2-2}\cline{3-3}
\lr{1}&\lr{4}\\
\cline{1-1}\cline{2-2}
\end{array}$}
}

& $\{(3, 4)\}$ &  1 & 12435 & 12435 $=s_3$ \\ \hline

21435 $=s_1 s_3$ & 21435 &  
{\def\lr#1{\multicolumn{1}{|@{\hspace{.6ex}}c@{\hspace{.6ex}}|}{\raisebox{-.3ex}{$#1$}}}
\raisebox{-.6ex}{$\begin{array}[b]{ccc}
\cline{1-1}\cline{2-2}\cline{3-3}
\lr{1}&\lr{3}&\lr{5}\\
\cline{1-1}\cline{2-2}\cline{3-3}
\lr{2}&\lr{4}\\
\cline{1-1}\cline{2-2}
\end{array}$}
}

& $\{(1, 2) , (3, 4)\}$&  2& 21435 & 21435 $=s_1 s_3$\\ \hline

 24135 $= s_1 s_3 s_2$ & 31425 &  
{\def\lr#1{\multicolumn{1}{|@{\hspace{.6ex}}c@{\hspace{.6ex}}|}{\raisebox{-.3ex}{$#1$}}}
\raisebox{-.6ex}{$\begin{array}[b]{ccc}
\cline{1-1}\cline{2-2}\cline{3-3}
\lr{1}&\lr{2}&\lr{5}\\
\cline{1-1}\cline{2-2}\cline{3-3}
\lr{3}&\lr{4}\\
\cline{1-1}\cline{2-2}
\end{array}$}
}
& $\{(2, 3) , (2, 4)\}$ &  2 & 13425 & 14235 $=s_3 s_2$ \\ \hline

12453 $=s_3 s_4$ &  12534   & 
{\def\lr#1{\multicolumn{1}{|@{\hspace{.6ex}}c@{\hspace{.6ex}}|}{\raisebox{-.3ex}{$#1$}}}
\raisebox{-.6ex}{$\begin{array}[b]{ccc}
\cline{1-1}\cline{2-2}\cline{3-3}
\lr{2}&\lr{3}&\lr{4}\\
\cline{1-1}\cline{2-2}\cline{3-3}
\lr{1}&\lr{5}\\
\cline{1-1}\cline{2-2}
\end{array}$}
}

&  $\{(4, 5)\}$ &  1 & 12354 & 12354 $=s_4$ \\ \hline

21453 $=s_3 s_4 s_1$ & 21534 &  
{\def\lr#1{\multicolumn{1}{|@{\hspace{.6ex}}c@{\hspace{.6ex}}|}{\raisebox{-.3ex}{$#1$}}}
\raisebox{-.6ex}{$\begin{array}[b]{ccc}
\cline{1-1}\cline{2-2}\cline{3-3}
\lr{1}&\lr{3}&\lr{4}\\
\cline{1-1}\cline{2-2}\cline{3-3}
\lr{2}&\lr{5}\\
\cline{1-1}\cline{2-2}
\end{array}$}
}
& $\{(1, 2) , (4, 5)\}$ &  2 & 21354 & 21354 $= s_4 s_1$ \\ \hline

 24153 $=s_3 s_4 s_1 s_2$ &  31524 &  
{\def\lr#1{\multicolumn{1}{|@{\hspace{.6ex}}c@{\hspace{.6ex}}|}{\raisebox{-.3ex}{$#1$}}}
\raisebox{-.6ex}{$\begin{array}[b]{ccc}
\cline{1-1}\cline{2-2}\cline{3-3}
\lr{1}&\lr{2}&\lr{4}\\
\cline{1-1}\cline{2-2}\cline{3-3}
\lr{3}&\lr{5}\\
\cline{1-1}\cline{2-2}
\end{array}$}
}
 & $\{(2, 3) , (4, 5)\}$ &  2 & 13254 & 13254 $=s_4 s_2$ \\ \hline

24513 $= s_3 s_4 s_1 s_2 s_3$ &  41523 &  
{\def\lr#1{\multicolumn{1}{|@{\hspace{.6ex}}c@{\hspace{.6ex}}|}{\raisebox{-.3ex}{$#1$}}}
\raisebox{-.6ex}{$\begin{array}[b]{ccc}
\cline{1-1}\cline{2-2}\cline{3-3}
\lr{1}&\lr{2}&\lr{3}\\
\cline{1-1}\cline{2-2}\cline{3-3}
\lr{4}&\lr{5}\\
\cline{1-1}\cline{2-2}
\end{array}$}
}
&  $\{(3, 4) , (3, 5)\}$ &  2 & 12453 &  12534 $ = s_4 s_3$ \\ \hline

\end{tabular}
\end{table}

We claim that, as in the $n=4$ case, the rolldowns satisfy
the condition~\eqref{eq:inequality-2}, which then implies by the same
argument that the corresponding Springer Schubert classes are
poset-upper-triangular and hence linearly independent and a module
basis. To prove this claim it suffices to check~\eqref{eq:inequality-2}
for those $w$ for which $\roll(w) \neq w$. We check each case by
hand.

For $w
= s_1 s_3 s_2$ with $\roll(w) = s_3 s_2$, we see that
$\roll(w) < s_3 s_4 s_1 s_2$ and $\roll(w) < s_3 s_4 s_1 s_2
s_3$. Since also 
$w < s_3 s_4 s_1 s_2$ and $w < s_3 s_4 s_1 s_2　s_3$, the
claim holds in this case.  
Next observe that the last four fixed points in the above table 
are linearly ordered with respect to the Bruhat order, i.e. 
\[
s_3 s_4 < s_3 s_4 s_1 < s_3 s_4 s_1 s_2 < s_3 s_4 s_1 s_2 s_3.
\]
In the case of $w=s_3 s_4$ we have $\roll(w) = s_4$. Moreover 
$s_4$ is not Bruhat-less than any of the fixed points occurring in the
$n=4$ case and is Bruhat-less than all of the last four fixed points,
so the claim holds in this case. Similarly, the rolldowns for the last
three fixed points satisfy
\[
\roll(s_3 s_4 s_1) = s_4 s_1 \not < s_3 s_4, \quad 
\roll(s_3 s_4 s_1 s_2) = s_4 s_2 \not < s_3 s_4 s_1, \quad 
\roll(s_3 s_4 s_1 s_2 s_3) = s_4 s_3 \not < s_3 s_4 s_1 s_2, 
\]
so the claim holds in all cases. This proves the claim and hence that
the Springer Schubert classes are poset-upper-triangular in the $n=5$ case and hence a
module basis, as desired.

\end{example}

\section{A poset pinball module basis for $(n-2,2)$ Springer
varieties}\label{sec:pinball}

The main result of this section is that the dimension pair algorithm
produces a set of Springer Schubert classes $\{p_{\roll(w)}\}_{w \in
  \mathcal{S}_{(n-2,2)}^{S^1}}$ which are a 
module basis, in the case of $(n-2,2)$ Springer varieties for any $n
\geq 4$. We have the following.

\begin{theorem}\label{theorem:pinball basis}
Let $n \geq 4$. Let $N:\C^n \to \C^n$ be a nilpotent operator in standard Jordan
canonical form with weakly decreasing Jordan block sizes $n-2$ and
$2$. 
Let 
$N_{hf} := \sigma N \sigma^{-1}$ be the choice of highest form of $N$ where
$\sigma$ is the permutation corresponding to the rotated English
filling of $(n-2,2)$. 
Let $\mathcal{S}_{(n-2,2)}$ be the 
Springer 
variety corresponding to $N_{hf}$ 
equipped with the $S^1$-action defined
in~\eqref{eq:conjugated circle}. 
Let $\roll: \mathcal{S}_{(n-2,2)}^{S^1} \to S_n$ be the function
defined by the dimension-pair
algorithm. Then the columns of the matrix 
\[
(p_{\roll(w)}(u))_{w, u \in \mathcal{S}_{(n-2,2)}^{S^1}}
\]
with entries in $H^*_{S^1}(\pt) \cong \C[t]$ are linearly independent
over $H^*_{S^1}(\pt)$.  (Here $w$ is the variable indexing the columns
and $u$ the index of the rows.) In particular, the Springer Schubert classes
$\{p_{\roll(w)}\}_{w \in \mathcal{S}_{(n-2,2)}^{S^1}}$ form a
$H^*_{S^1}(\pt)$-module basis for the equivariant cohomology ring
$H^*_{S^1}(\mathcal{S}_{(n-2,2)})$ of the Springer variety.
\end{theorem}

\begin{remark}
  The above theorem extends the subregular Springer case (which corresponds to
  Young diagrams of shape $(n-1,1)$), for which it was shown in
  \cite{HarTym10} that the set of Springer Schubert classes obtained
  by the dimension pair algorithm is poset-upper-triangular, so in
  particular linearly independent.  (Although the results in
  \cite{HarTym10} are not phrased using the terminology of this paper it is
  straightforward to see that the classes used in \cite{HarTym10} agree
  with those arising from the dimension pair algorithm. )
\end{remark}

Since the rows are increasing in a Springer permissible filling, we
can naturally decompose the set of $(n-2,2)$ permissible fillings into
two subsets: namely, those for which the largest entry $n$ occupies
the top row, and those for which $n$ occupies the bottom row. As
observed in Example~\ref{example:n=5} above, when $n$ is in the top
row, the permissible filling obtained by removing the rightmost
box in the top row is a permissible filling for the Young diagram
$(n-3,2)$, corresponding to the smaller Springer variety
$\mathcal{S}_{(n-3,2)}$.  This sets us up for an inductive
argument. Since we have already seen in Section~\ref{sec:examples} the
linear independence for the cases $n=4$ and $n=5$, we start the induction
start at $n=6$.  We begin with a preliminary lemma
generalizing the observations made in Example~\ref{example:n=5}.

\begin{lemma}\label{lemma:rolldown table for fixed points}
Let $n \geq 6$. Let $N, N_{hf}, \mathcal{S}_{(n-2,2)}$ and $\roll$ be
as in Theorem~\ref{theorem:pinball basis}.
Then 
\begin{itemize}
\item there are precisely $n-1$ permissible fillings of $(n-2,2)$ with
  $n$ in the bottom row, 
\item the $n-1$ such permissible fillings, their corresponding
  Springer fixed points $w$, and their rolldowns $\roll(w)$ are
  precisely those listed in the table below, 
\item these $n-1$ Springer fixed points are linearly ordered with
  respect to Bruhat order, i.e. 
  \begin{equation}
    \label{eq:total order}
    s_3 s_4 \cdots s_{n-2} s_{n-1} < s_3 s_4 \cdots s_{n-2} s_{n-1}
    s_1 < \cdots < s_3 s_4 \cdots s_{n-2} s_{n-1} s_1 s_2 \cdots
    s_{n-3} s_{n-2}.
  \end{equation}
\end{itemize}

\begin{table}[h] \centering
\caption{Dimension pair algorithm data for the Springer fixed points in $\mathcal{S}_{(n-2,2)}^{S^1}$
  corresponding to permissible fillings with $n$ in the bottom row.}
\begin{tabular}{ | c | c | c | c |}\hline
 pf & $w^{-1}$ & $w$ &  $v_{hf}$ \\ \hline \hline
{\def\lr#1{\multicolumn{1}{|@{\hspace{.6ex}}c@{\hspace{.6ex}}|}{\raisebox{-.3ex}{$#1$}}}
\raisebox{-.6ex}{$\begin{array}[b]{ccccc}
\cline{1-1}\cline{2-2}\cline{3-3}\cline{4-4}\cline{5-5}
\lr{2}&\lr{3}&\lr{4}& \lr{\cdots} & \lr{n-1}\\
\cline{1-1}\cline{2-2}\cline{3-3}\cline{4-4}\cline{5-5}
\lr{1}&\lr{n}\\
\cline{1-1}\cline{2-2}
\end{array}$}
}
& $1 \hsm 2 \hsm n \hsm 3 \hsm 4 \cdots$
& $s_3 s_4 \cdots s_{n-2}
    s_{n-1}$ 
& $s_{n-1}$ \\ \hline
{\def\lr#1{\multicolumn{1}{|@{\hspace{.6ex}}c@{\hspace{.6ex}}|}{\raisebox{-.3ex}{$#1$}}}
\raisebox{-.6ex}{$\begin{array}[b]{ccccc}
\cline{1-1}\cline{2-2}\cline{3-3}\cline{4-4}\cline{5-5}
\lr{1}&\lr{3}&\lr{4}& \lr{\cdots} & \lr{n-1}\\
\cline{1-1}\cline{2-2}\cline{3-3}\cline{4-4}\cline{5-5}
\lr{2}&\lr{n}\\
\cline{1-1}\cline{2-2}
\end{array}$}
}
& $2 \hsm 1 \hsm n \hsm 3 \hsm 4 \cdots $
& $s_3 s_4
    \cdots s_{n-2} s_{n-1} s_1$
&  $s_{n-1}s_1$ \\  \hline
{\def\lr#1{\multicolumn{1}{|@{\hspace{.6ex}}c@{\hspace{.6ex}}|}{\raisebox{-.3ex}{$#1$}}}
\raisebox{-.6ex}{$\begin{array}[b]{ccccc}
\cline{1-1}\cline{2-2}\cline{3-3}\cline{4-4}\cline{5-5}
\lr{1}&\lr{2}&\lr{4}& \lr{\cdots} & \lr{n-1}\\
\cline{1-1}\cline{2-2}\cline{3-3}\cline{4-4}\cline{5-5}
\lr{3}&\lr{n}\\
\cline{1-1}\cline{2-2}
\end{array}$}
}
& $3 \hsm 1 \hsm n \hsm 2 \hsm 4 \cdots$
& $s_3 s_4 \cdots s_{n-2}
    s_{n-1} s_1 s_2$ 
&  $s_{n-1}s_2$\\ \hline
{\def\lr#1{\multicolumn{1}{|@{\hspace{.6ex}}c@{\hspace{.6ex}}|}{\raisebox{-.3ex}{$#1$}}}
\raisebox{-.6ex}{$\begin{array}[b]{ccccc}
\cline{1-1}\cline{2-2}\cline{3-3}\cline{4-4}\cline{5-5}
\lr{1}&\lr{2}&\lr{3}& \lr{\cdots} & \lr{n-1}\\
\cline{1-1}\cline{2-2}\cline{3-3}\cline{4-4}\cline{5-5}
\lr{4}&\lr{n}\\
\cline{1-1}\cline{2-2}
\end{array}$}
}
& $4 \hsm 1 \hsm n \hsm 2 \hsm 3 \cdots $
& $s_3 s_4 \cdots s_{n-2} s_{n-1} s_1 s_2
    s_3$
 & $s_{n-1}s_3$\\  \hline
\vdots &  \vdots & \vdots & \vdots \\ \hline
{\def\lr#1{\multicolumn{1}{|@{\hspace{.6ex}}c@{\hspace{.6ex}}|}{\raisebox{-.3ex}{$#1$}}}
\raisebox{-.6ex}{$\begin{array}[b]{ccccc}
\cline{1-1}\cline{2-2}\cline{3-3}\cline{4-4}\cline{5-5}
\lr{1}&\lr{2}&\lr{\cdots}& \lr{n-3} & \lr{n-2}\\
\cline{1-1}\cline{2-2}\cline{3-3}\cline{4-4}\cline{5-5}
\lr{n-1}&\lr{n}\\
\cline{1-1}\cline{2-2}
\end{array}$}
}
& $n-1 \hsm 1 \hsm n \hsm 2 \hsm 3 \cdots$
& $s_3 s_4 \cdots s_{n-2} s_{n-1}
    s_1 s_2 \cdots s_{n-3} s_{n-2}$
&  $s_{n-1}s_{n-2}$ \\ \hline
\end{tabular}
\end{table}

\end{lemma}

\begin{proof}[Proof of Lemma~\ref{lemma:rolldown table for fixed points}]
Since the Springer permissible fillings are precisely those which are
row-strict, it is immediate that the permissible fillings listed in
the table are precisely those with $n$ in the bottom row. In
particular there are exactly $n-1$ such permissible fillings as
claimed. Moreover, it follows
from the
definition of $\phi_{\lambda,\sigma}$ (which corresponds to the
rotated English reading) 
that the one-line notation of the $w^{-1}$ are those given in the
table.
Explicit computation also verifies that the following expressions
in the simple transpositions are indeed reduced word decompositions of
the $w^{-1}$: 
\begin{itemize}
\item $1 \hsm 2 \hsm n \hsm 3 \hsm 4 \cdots = s_{n-1} s_{n-2} \cdots
  s_4 s_3$
\item $2 \hsm 1 \hsm n \hsm 3 \hsm 4 \cdots = s_1 s_{n-1} s_{n-2} \cdots
s_4 s_3$  
\item  $3 \hsm 1 \hsm n \hsm 2 \hsm 4 \cdots = s_2 s_1 s_{n-1} s_{n-2}
\cdots s_4 s_3 $   
\item $4 \hsm 1 \hsm n \hsm 2 \hsm 3 \cdots = s_3 s_2 s_1 s_{n-1} s_{n-2}
\cdots s_4 s_3$  
\item $\vdots$
\item  $n-1 \hsm 1 \hsm n \hsm 2 \hsm 3 \cdots = s_{n-2} s_{n-1} \cdots s_2
s_1 s_{n-1} s_{n-2} \cdots s_4 s_3 $, 
\end{itemize}
from which it follows that the $w$ are those given in the
list. For $k$ with $1 \leq k \leq n-1$, the definition of dimension pairs implies that the permissible
filling with $k$ and $n$ in the bottom row contains as dimension pairs
$\{(1,k), (n-1,n)\}$ for $2 \leq k \leq n-1$ and $\{(n-1,n)\}$ for
$k=1$. From this it follows from the definition of $\omega(\mathbf{x})$
that $\roll(w)$ is as given in
the table. Finally, from the given reduced word decompositions and the
definition of Bruhat order we obtain~\eqref{eq:total order} as
desired. 
\end{proof}

Before proceeding with the proof of Theorem~\ref{theorem:pinball
  basis} we briefly
recall the 
\textbf{Billey formula} for computing restrictions $\sigma_v(w)$ of 
Schubert classes $\sigma_v$ at some $w$ in $S_n$. 
We use the formulation given in \cite{Knu03}. 
Let $\alpha_i$ denote the simple root $t_i - t_{i+1}$
and $\widehat{\alpha}_i$ the operator on $H^*_T(\pt)$ which multiplies
by $\alpha_i$. 

\begin{theorem}\label{theorem:sara billey} \textbf{(\cite[Theorem 4]{Bil99}, also
  cf. \cite{Knu03})}
Suppose $I$ is a reduced word
expression for \(w \in S_n.\) For each \(v \in S_n\) we have 
\begin{equation}\label{eq:billey}
\sigma_v(w) = \sum_{J \subseteq I} \prod_{i \in I} 
  \left(\widehat{\alpha_i}^{[i \in J]} r_i \right) \cdot 1 
\end{equation}
where the sum is over reduced subwords $J$ of $I$ with product $v$, 
the notation $\widehat{\alpha}_i^{[i \in J]}$ means that
$\widehat{\alpha}_i$ is included only if
$i \in J$, and $r_i$ is the reflection corresponding to $s_i$.
\end{theorem}

We record the following fact, used in the proof below, which follows
straightforwardly from the
Billey formula. 

\begin{fact}\label{fact:extra}
  Suppose $v, w \in S_n$ with $v \leq w$ in Bruhat order. Suppose
  there exists a decomposition $w = w' \cdot w''$ for $w', w'' \in
  S_n$ where $v \leq w'$ and, for all simple transpositions $s_i$ such
  that $s_i < v$, we have $s_i \not \leq w''$. Then $\sigma_v(w) =
  \sigma_v(w')$. 
\end{fact}

As explained in Section~\ref{sec:betti}, we need to compute
the restrictions $p_{\roll(w)}(u)$ for $w, u$ Springer fixed
points. Since $p_{\roll(w)}$ is by definition the image of 
the equivariant Schubert class $\sigma_{\roll(w)}$ under the ring
map~\eqref{eq:projection} and because the diagram 
\[
\xymatrix{
H^*_{T}(\Flags(\C^n)) \ar @{^{(}->}[r] \ar[d] & H^*_{T}((\Flags(\C^n))^{T}) \cong
\bigoplus_{w \in W} H^*_{T}(\pt) \ar[d] \\
H^*_{S^1}(\mathcal{S}_{(n-2,2)}) \ar @{^{(}->}[r] &
H^*_{S^1}(\mathcal{S}_{(n-2,2)})^{S^1}) \cong \bigoplus_{w \in \Hess(h)^{S^1}}
H^*_{S^1}(\pt) 
}
\]
commutes, the polynomial $p_{\roll(w)}(u) \in H^*_{S^1}(\pt) \cong
\C[t]$ can be computed by first evaluating $\sigma_{\roll(w)}(u)$ by
the Billey formula~\eqref{eq:billey} and then using the linear projection $\t^*
\to \Lie(S^1)^*$ for our choice of $S^1$ in~\eqref{eq:conjugated
  circle} given in Lemma~\ref{lemma:conjugated circle}. We use this
technique repeatedly in the proof below.

\begin{proof}[Proof of Theorem~\ref{theorem:pinball basis} ]

By Propositions~\ref{prop:linear independence} and~\ref{prop: matrix
  form}, it suffices to prove that the matrix obtained from the
restrictions to fixed points 
\[
(p_{\roll(w)}(u))_{w, u \in \mathcal{S}_{(n-2,2)}^{S^1}}
\]
has $H^*_{S^1}(\pt)$-linearly independent columns. 

Let $n \geq 4$. We have seen in Section~\ref{sec:examples} that the
above assertion holds for the cases $n=4$ and $n=5$.
Hence assume now that $n \geq 6$. We assume by induction that for the 
$n-1$ case, i.e. for the case of the partition $(n-3,2)$, 
the above matrix has linearly independent columns.

For concreteness and for the remainder of the argument, we assume that
the fixed points $w \in \mathcal{S}_{(n-2,2)}^{S^1}$ have been
linearly ordered so that the fixed points corresponding to permissible
fillings containing the $n$ in the top row appear first, and that the
fixed points associated to fillings with $n$ in the bottom row are
given the ordering in the table in Lemma~\ref{lemma:rolldown
  table for fixed points} (reading from top to bottom). Ordered in
this manner, we may write the above matrix in terms of submatrices as
follows:
\begin{equation}\label{eq:ABCD}
(p_{\roll(w)}(u))_{w, u \in \mathcal{S}_{(n-2,2)}^{S^1}}
= \begin{bmatrix} A & B \\ C & D \end{bmatrix}
\end{equation}
where the submatrix $A$ has entries 
$p_{\roll(w)}(u)$ where both $w, u$ correspond to fillings with $n$ in
the top row, $D$ corresponds to those where both $w,u$ have $n$ in the
bottom row, and so on.

Consider the submatrix $A$. 
For an entry $p_{\roll(w)}(u)$ in $A$, by assumption $w$ is in the
subgroup $S_{n-1} \subseteq S_n$ and it is straightforward to see from
the definition of the dimension pair algorithm that $\roll(w)$ is
equal to the rolldown of $w$ considered as an element of
$\mathcal{S}^{S^1}_{(n-3,2)}$. Since $u \in S_{n-1}$ also this
submatrix is equal to the matrix of restrictions to fixed points
obtained in the $(n-3,2)$ case and so by induction $A$ has
linearly independent columns. 

Next consider the submatrix $B$ corresponding to $p_{\roll(w)}(u)$
where $\phi_{\lambda,\sigma}^{-1}(w^{-1})$ has $n$ 
in the bottom row and $\phi_{\lambda,\sigma}^{-1}(u^{-1})$ has $n$ in the top row. From
Lemma~\ref{lemma:rolldown table for fixed points} and the table given
there, 
we know that the rolldown $\roll(w)$ of any such
$w$ contains the simple transposition $s_{n-1}$ in its reduced word
decomposition. On the other hand, for $u$ with $n$ in the top row,
$u$ is an element in the subgroup $S_{n-1}$ which fixes the element
$n$, and in particular a reduced word decomposition for $u$ may be
written solely with the simple transpositions $s_1, s_2, \ldots,
s_{n-2}$. Hence $\roll(w) \not \leq u$ in
Bruhat order, and by the upper-triangularity
property~\eqref{eq:schuberts upper triangular} of equivariant
Schubert classes this implies
$p_{\roll(w)}(u)=0$. We 
conclude that the entire submatrix is $0$ and 
the matrix~\eqref{eq:ABCD} is in fact of the form
\[
\begin{bmatrix} 
A & 0 \\
C & D 
\end{bmatrix} 
\]
where $A$ has linearly independent columns. In order to prove that the
full matrix has linearly independent columns, we wish to prove
that the submatrix $D$ has linearly independent columns. The remainder
of the proof is dedicated to the justification of this last claim, for
which we explicitly compute the appropriate entries $p_{\roll(w)}(u)$
using the Billey formula~\eqref{eq:billey}.

We compute each column of $D$ in the linear order 
given by the enumeration in Lemma~\ref{lemma:rolldown table for fixed
  points} of those $w$ with $n$ in the bottom row. For the
Billey computations below we use the choices of reduced word
decompositions for $w$ and $\roll(w)$ 
given in the same lemma.

First consider the case \(w = s_3 s_4 \cdots s_{n-2} s_{n-1}.\) Then \(\roll(w) = s_{n-1}.\) 
We claim $\sigma_{s_{n-1}}$ evaluates to $t_3 - t_n$ at all
fixed points $u$. Indeed, recalling that the reflection $r_i$ acts on the
variables $t_j$ by $r_i(t_i) = t_{i+1}, r_i(t_{i+1}) = t_i$ and
$r_i(t_j) = t_j$ for all $j \neq i, i+1$, we have for instance 
\begin{equation}\label{eq:first w}
\begin{split}
\sigma_{s_{n-1}}(s_3 s_4 \cdots s_{n-2} s_{n-1}) & = r_3 r_4 \cdots
r_{n-2} (t_{n-1} - t_n) \\
 & = r_3 r_4 \cdots r_{n-3} (t_{n-2} - t_n) \\
 &  \vdots \\ 
 & = t_3 -  t_{n},
\end{split}
\end{equation}
which proves the claim for $u = w$. For all other 
$u$ with $n$ in the bottom row, the computation of the Billey
formula differs from~\eqref{eq:first w} only in that there
are extra simple transpositions occurring after the $s_{n-1}$ in the
reduced word decomposition of $u$. By Fact~\ref{fact:extra}
these extra transpositions make
no difference in the evaluation of $\sigma_{s_{n-1}}(u)$ and so 
$\sigma_{s_{n-1}}(u) = t_3 - t_n$ for all $u$.
The restriction $p_{\roll(w)}(u) = p_{s_{n-1}}(u)$ is equal to the
image of $\sigma_{s_{n-1}}(u) \in H^*_{T}(\pt)$ under the projection
map $H^*_{T}(\pt) \to H^*_{S^1}(\pt)$ induced from the inclusion
$S^1 \into T$. By Lemma~\ref{lemma:conjugated circle} we know $t_3 \mapsto t$ and
$t_n \mapsto (n+3-n)t = 3t$ under this projection, 
from which we conclude that the first (leftmost) column of $D$ is 
\[
\begin{bmatrix} (t -3t) = -2t \\ -2t \\ \vdots \\ -2t \end{bmatrix}.
\]

Next consider the case \(w = s_3 s_4 \cdots s_{n-2} s_{n-1} s_1\) and \(\roll(w) = s_{n-1}s_1.\) 
In this case, \(\sigma_{s_{n-1}s_1}(s_3 s_4 \cdots s_{n-2} s_{n-1}) =
0\) since $s_{n-1}s_1$ does not occur as a subword of $s_3 s_4 \cdots
s_{n-2} s_{n-1}$. Also, 
$\sigma_{s_{n-1}s_1}$ evaluates to $(t_3 - t_{n})(t_1 - t_2)$
  at all other $u$. This can be seen from the computation
\begin{equation}\label{eq:next w}
\begin{split}
\sigma_{s_{n-1}s_1}(s_3 s_4 \cdots s_{n-2} s_{n-1} s_1) & = (r_3 r_4
\cdots r_{n-2}(t_{n-1}-t_n))(r_3 r_4 \cdots r_{n-2} r_{n-1}(t_1 -
t_2)) \\
 & = (r_3 r_4
\cdots r_{n-2}(t_{n-1}-t_n))(t_1 - t_2) \\
& = (t_3 -t_{n})(t_1 - t_2)
\end{split}
\end{equation}
for the case $w = s_3 s_4 \cdots s_{n-2} s_{n-1} s_1$. The computation
at other $u$
follows from~\eqref{eq:next w} and Fact~\ref{fact:extra}. Applying
Lemma~\ref{lemma:conjugated circle} 
again we obtain that the column corresponding to this $w$ is 
\[
\begin{bmatrix} 0 \\ 2(n-2)t^2 \\ 2(n-2)t^2 \\ \vdots \\ 2(n-2)t^2 \end{bmatrix}.
\]

Next consider the case \(w = s_3 s_4 \cdots s_{n-2} s_{n-1} s_1 s_2\) and \(\roll(w) =
  s_{n-1} s_2.\) 
In this case
\[
\sigma_{s_{n-1}s_2}(s_3 s_4 \cdots s_{n-2} s_{n-1}) =
  \sigma_{s_{n-1}s_2}(s_3 s_4 \cdots s_{n-2} s_{n-1} s_1) = 0
\]
since
  there are no reduced subwords in $s_3 s_4 \cdots s_{n-2} s_{n-1}$
  equal to $\roll(w) = s_{n-1}s_2$. 
Furthermore, $\sigma_{s_{n-1}s_2}$ evaluates to $(t_1-t_4)(t_3 - t_{n})$ on
  all other $u$. Since the computations are similar to
  those given above we henceforth keep explanation brief. We have 
\[
\sigma_{s_{n-1}s_2}(s_3 s_4 \cdots s_{n-1}s_1 s_2) = (t_3 -
t_{n})(t_1 - t_4)
\]
and at other $u$ the computation is similar. 
Hence the column corresponding to this $w$ is 
\[
\begin{bmatrix} 0 \\ 0 \\ 2(n-3)t^2 \\ \vdots \\ 2(n-3)t^2 \end{bmatrix}.
\]

Next consider the case \(w = s_3 s_4 \cdots s_{n-2} s_{n-1} s_1 s_2
  s_3\) where $\roll(w)$ is $s_{n-1}s_3$.
In this case, by arguments similar to those above, 
$\sigma_{s_{n-1}s_3}$ evaluates to $(t_3 - t_4)(t_3 - t_{n})$ on
  the first $3$ fixed points listed in the table in
  Lemma~\ref{lemma:rolldown table for fixed points}. 
Moreover $\sigma_{s_{n-1}s_3}$ evaluates to $(t_3 - t_4)(t_3 - t_{n}) +
  (t_1 - t_5)(t_3 - t_{n})$ at all other $u$. 
We conclude the column corresponding to this $w$ is 
\[
\begin{bmatrix} 2n t^2  \\ 2nt^2 \\ 2n t^2 \\ 4(n-1)t^2 \\ \vdots \\ 4(n-1)t^2 \end{bmatrix},
\]
where there are $(n-1)-3 =n-4$ entries of the form $4(n-1)t^2$. 

For the next case, suppose $n \geq 7$. (In the special case $n=6$, this case is
  vacuous.) Suppose $k \in \Z$ with $4 \leq k \leq n-3$. 
  Let \(w = s_3 s_4 \cdots s_{n-2} s_{n-1} s_1
  s_2 \cdots s_{k-1}s_k\) and \(\roll(w) = s_{n+1} s_k.\) 
By assumption on $k$, the simple transposition $s_k$ commutes with $s_{n-1}$. 
In this case $\sigma_{s_{n-1}s_k}$ evaluates to $(t_3 - t_{k+1})(t_3 -
  t_{n})$ on all fixed points listed in Lemma~\ref{lemma:rolldown
    table for fixed points} up to $s_3 s_4 \cdots s_{n-2} s_{n-1} s_1
  s_2 \cdots s_{k-1}$. There are $k$ fixed points in all of this
  form. 
Moreover, $\sigma_{s_{n-1}s_k}$ evaluates to $(t_3 - t_{k+1})(t_3 -
  t_{n}) + (t_1 - t_{k+2})(t_3 - t_{n})$ on the remaining fixed
  points $u$ which contain $s_3 s_4 \cdots s_{n-2} s_{n-1} s_1 s_2 \cdots
  s_{k-1}s_k$. 
Hence when projected to $H^*_{S^1}(\pt)$, the column corresponding to such a $w$ is 
\[
\begin{bmatrix} 2(n-k+3) t^2 \\ \vdots \\ 2(n-k+3)t^2 \\ ( 2(n-k+3) +
  2(n-k+1) ) t^2 \\ \vdots \\ ( 2 (n-k+3) + 2(n-k+1) ) t^2 \end{bmatrix}
\]
where there are $k$ entries of the form $2(n-k+3)t^2$ and $n-1-k$
entries of the form $( 2 (n-k+3) + 2(n-k+1) ) t^2$. 

Finally, consider the case $w = s_3 s_4 \cdots s_{n-2} s_{n-1} s_1
  s_2 \ldots s_{n-3} s_{n-2}$ and $\roll(w) = s_{n-1} s_{n-2}$. Since 
  $s_{n-1}$ and $s_{n-2}$ do not commute, this computation is somewhat
  different
  from the ones given above; 
  in particular $\roll(w)$ is \emph{not} Bruhat-less than any of the fixed
  points $u$ except for the last
  one listed in Lemma~\ref{lemma:rolldown table for fixed points}.
Hence in this case
$\sigma_{s_{n-1}s_{n-2}}(u) = 0$ at all $u$ except for $u=s_3
  s_4 \cdots s_{n-2} s_{n-1} s_1 s_2 \cdots s_{n-3} s_{n-2}$, and 
at this last $u$, 
we can compute 
\[
\sigma_{s_{n-1} s_{n}}( s_3 s_4 \cdots s_{n-2} s_{n-1} s_1 s_2 \cdots
s_{n-3} s_{n-2}) = (t_3 - t_{n})(t_1 - t_{n}).
\]
Hence the column corresponding to this last $w$ is 
\[
\begin{bmatrix} 0 \\ 0 \\ \vdots \\ 0 \\ 2t^2\end{bmatrix}.
\]

We now prove that the columns $p_w$ for $w$ as above are linearly
independent over the ring $H^*_{S^1}(\pt) \cong \C[t]$. The first
column $p_w$ with $w = s_3 s_4 \cdots s_{n-2} s_{n-1}$ has a $-2t$ in
each entry. We may add or subtract any
multiple of this column to or from any other column, and if the
resulting set of columns is 
linearly independent, then so is the original set of
columns. 
It is straightforward to check that for all $k$ with $3
\leq k \leq n-1$, subtracting $2(n-k+3)$ times the first column from
the column corresponding to $w$ with rolldown $\roll(w) = s_{n-1}s_k$
yields
\[
\begin{bmatrix}
  0 \\
\vdots \\
0 \\
2(n-k+1) t^2 \\
\vdots \\
2(n-k+1) t^2 
\end{bmatrix}
\]
where there are $k$ zeroes at the top of the column and $(n-1)-k$
entries at the bottom of the form $2(n-k+1)t^2$.
In particular, adjusted in this manner, the resulting matrix
is lower-triangular with non-zero entries along the diagonal, so 
its columns are linearly independent. As argued above, this implies
that the matrix $D$ has linearly independent columns, as was desired. 
This completes the proof. 
\end{proof}

\section{Open questions}\label{sec:open questions}

We close with some open questions for future work. 

\begin{question}
  The computations in the proof of Theorem~\ref{theorem:pinball basis} explicitly show that the set of
  classes $\{p_{\roll(w)}\}_{w \in \mathcal{S}_{(n-2,2)}^{S^1}}$ are
  not poset-upper-triangular for $n \geq 6$ since the submatrix $D$
  discussed in the proof has non-zero entries both above and below its
  main diagonal. However the proof also shows that a simple change of
  basis does yield a poset-upper-triangular basis. We do not know
  whether this is an instance of a more general phenomenon. 
It would be of interest to clarify the situation for other cases of
Springer varieties. 
\end{question}

\begin{question}
Both
Tymoczko's paving by affines of Hessenberg varieties and the
interpretation of our dimension pair algorithm via poset pinball
depend on using a Hessenberg variety $\Hess(N,h)$ for which the
nilpotent operator $N$ is in highest form. In the case of Tymoczko's
paving, this choice can be viewed as a matter only of convenience in
the sense that any other translated Hessenberg variety $\Hess(\sigma N
\sigma^{-1}, h)$ can be given a paving simply by using translated
Schubert cells $\sigma \cdot BwB$ instead of the usual Schubert cells
$BwB$. On the other hand, the poset pinball game delicately depends on the
choice of initial subset \[\Hess(N,h)^{S^1} \subseteq S_n.\]
 Although
the sets $\Hess(N,h)^{S^1}$ and $\Hess(\sigma N\sigma^{-1},h)^{S^1}$
are also related by a simple translation by $\sigma$, multiplication
by a permutation does \emph{not} preserve Bruhat order, so pinball
results do not immediately translate from $\Hess(N,h)$ to
$\Hess(\sigma N\sigma^{-1},h)$. One of the main results of this
manuscript is that, for a certain special family of Hessenberg
varieties $\Hess(N,h) = \mathcal{S}_{N_{hf}}$ (where $N_{hf}$ is a
particular choice of highest form) we can use the poset pinball and
the dimension pair algorithm to obtain a module basis for
$H^*_{S^1}(\mathcal{S}_{N_{hf}})$. 

\begin{enumerate}
\item It seems plausible that there may
be other choices of highest forms
(cf. Theorem~\ref{theorem:N_T and algorithm}), different from
that used in this manuscript, which are particularly well-suited for
poset pinball. 
\item Furthermore, among the choices of highest forms which
behave well for poset pinball, there may also be choices
best suited for further applications of pinball bases. More
specifically, 
there may be choices highest forms $N_T$
such that a pinball basis for 
$H^*_{S^1}(\mathcal{S}_{N_T})$ has good 
properties when mapped to $H^*_{S^1}(\mathcal{S}_{N_T}^{S^1})$. 
Such choices could then prove useful for 
e.g.
constructions of representations on equivariant cohomology (analogous
to the lifts of the classical Springer representations constructed via
pinball in \cite{HarTym10}). 
\end{enumerate}
\end{question}

\def\cprime{$'$}

\end{document}